\documentclass[12pt]{amsart}
\usepackage[english]{babel}
\usepackage[utf8]{inputenc}
\usepackage{lipsum}
\usepackage{mathrsfs}
\usepackage{stix}
\usepackage{fullpage}
\usepackage{mathtools}
\usepackage{amsmath}
\usepackage{tikz}
\usepackage{tikz-cd}
\usetikzlibrary{decorations.pathreplacing,angles,quotes}
\usetikzlibrary{arrows,chains,positioning,scopes,quotes}
\newtheorem{defi}{Definition}[section]
\newtheorem{exa}[defi]{Example}

\newtheorem{lema}[defi]{Lemma}
\newtheorem{teo}[defi]{Theorem}
\newtheorem{rem}[defi]{Remark}

\newtheorem{coro}[defi]{Corollary}
\newtheorem{pro}[defi]{Proposition}
\newtheorem*{rem*}{Remark}

\newcommand{\1}{\mathbb{1}}

\newcommand{\p}{{\mathfrak{p}}}

\newcommand{\interior}[1]{%
  {\kern0pt#1}^{\mathrm{o}}%
}
\newcommand{\MO}{\mathfrak{L}}

\newcommand{\C}{\mathbb{C}}
\newcommand{\Q}{\mathbb{Q}}
\newcommand{\K}{\mathbb{K}}
\newcommand{\R}{\mathbb{R}}
\newcommand{\N}{\mathbb{N}}
\newcommand{\Z}{\mathbb{Z}}

\newcommand{\esp}{\text{  }}

\usepackage[colorlinks=false]{hyperref}
\renewcommand\eqref[1]{(\ref{#1})} 
\begin{document}
\title[FUNDAMENTAL SOLUTION OF THE VLADMIROV-TAIBLESON OPERATOR ON NONCOMMUATIVE VILENKIN GROUPS] 
 {FUNDAMENTAL SOLUTION OF THE VLADMIROV-TAIBLESON OPERATOR ON NONCOMMUTATIVE VILENKIN GROUPS}

\author{Julio Delgado \and J.P. Velasquez-Rodriguez 
}

\newcommand{\Addresses}{

{
  \bigskip
  \footnotesize
  
  J.P.~VELASQUEZ-RODRIGUEZ, \textsc{Department of Mathematics: Analysis, Logic and Discrete Mathematics, Ghent University, Belgium.}\par\nopagebreak
  \textit{E-mail address:} \texttt{juanpablo.velasquezrodriguez@ugent.be}
}

{
  \bigskip
  \footnotesize
  
  JULIO DELGADO, \textsc{Department of Mathematics, Universidad del Valle, Colombia.}\par\nopagebreak
  \textit{E-mail address:} \texttt{delgado.julio@correounivalle.edu.co}
}

}

\thanks{The authors are supported by the FWO Odysseus 1 grant G.0H94.18N: Analysis and Partial Differential Equations. }

\subjclass[20é0]{Primary; 22E35, 35S05; Secondary: 12H25, 35A08. }


\date{\today}
\begin{abstract}
The fundamental solution and the heat semigroup of the Vladimirov-Taibleson operator on constant-order noncommutative Vilenkin groups are obtained, together with some estimates on the associated heat kernel. We also show the existence of a fundamental solution for the "Vladimirov Laplacian" on the $p$-adic Heisenberg group and the $p$-adic Engel group, and discuss possible extensions of our results to more general homogeneous operators on graded $p$-adic Lie groups.     
\end{abstract}
\maketitle
\tableofcontents

\section*{Introduction }
During the last 35 years there has been a growing interest in $p$-adic pseudo-differential equations, $p$-adic analysis and their connections to several disciplines like medicine, biology, physics, etc. See \cite{ap-adicmathfirst30years} and the references there in for a detailed account on some of the developments in $p$-adic analysis and applications during the last 35 years. In the same way many natural phenomena can be modelled by using complex numbers and differential equations, like the heat equation or the Schr{\"o}dinger equation, one can model some phenomena like fluid dynamics in porous media  \cite{appwaveletsreactiondiffusionmodels, e18070249} and protein dynamics \cite{Vladimirovbook} with ultrametric spaces and $p$-adic pseudo-differential equations. A very important tool in ultrametric analysis is the Vladimirov-Taibleson operator studied in \cite{Taiblesonbook, Vladimirovbook} which can be regarded as the analogue of the fractional Laplace operator on totally disconnected spaces. Equations with the Vladimirov-Taibleson operator have been studied intensively in the last few years, and one can find in the literature the study of its fundamental solutions \cite{RodriguezJ}, radial solutions \cite{Kochubei2014RadialSO}, an analogue of the Navier-Stokes equation \cite{p-adicnavierstokes}, non-linear equations and $p$-adic wavelets \cite{nonlinearp-adicequationsandwavelets}, a $p$-adic analogue of the porous medium equation \cite{porousmediumequation}, among many other works. So, in the same way the Laplacian plays a fundamental roll in physics and geometry, the Vladimirov-Taibleson operator is central in the study of nonarchimedean pseudo-differential equations, and the modelling of natural phenomena with ultrametric spaces, mostly done on a non-archimedean local field like the field of $p$-adic numbers $\Q_p$, or more generaly on abelian Vilenkin groups \cite{GreenFunctVladimirovDerivatives, UltraPseudoDiffEq, BookZG}. In the noncommutative case the literature about analysis on $p$-adic Lie groups and Vilenkin groups is rather limited and, in the knowledge of the authors, there are no works on the Vladimirov-Taibleson operator or pseudo-differential operators. We intend to contribute to the literature by computing the fundamental solution of this and some other related operators on noncommutative Vilenkin groups, a class of locally compact topological groups named in analogy with the well-known locally compact abelian Vilenkin groups, which is a class of topological groups named by Onneweer after Vilenkin who studied them first \cite{Vilenkinpaper}, even though the modern definition is probably due to Spector, Edwards and Gaudry, who called these groups \emph{``groups having a suitable family of open subgroups”}. See Onneweer's memory about his involvement with the dyadic derivative in \cite[Volume 1, Chapter 7]{Dyadic}. \\

The problem of the existence of a fundamental solution for pseudo-differential operators in ultrametric spaces is a non-trivial one and it has received a lot of attention in the literature. It is interesting, among other reasons, because it lies in the intersection of several areas, and because it is connected to some physical theories and the modelling of some natural phenomena. For example we know that:
\begin{itemize}
    \item There are several models that
describe the dynamics of hierarchic complex systems in terms of a master p-adic pseudo-differential equation, like the heat equation with the Vladimirov operator, see \cite{avetisov}. In this context the Cauchy problem associated to a pseudo-differential operator becomes relevant and it is natural to study its fundamental solution. See for instance \cite{Chuong2008TheCP, Torresblanca-Badillo2, RodriguezJ, fundsoltaibleson, Torresblanca-Badillo1} and the references therein. 
\item Vladimirov derivatives play an important role in the formulation of conformal field theories defined on a non-archimedean local field. In such setting the \emph{action} of the system can be described in terms of a kinetic term defined in terms of the Vladimirov derivative corresponding to a given Hecke character. See \cite{GreenFunctVladimirovDerivatives, doi:10.1142/S0217751X89002065, cmp/1104178891}. 
\item There is a connection between local zeta functions and fundamental solutions for pseudo-differential operators with polynomial symbols over non-archimedean local fields, see for instance \cite{articleFundSol}. It turns out that, for pseudo-differential operators with polynomial symbols, the existence of a meromorphic continuation for the local zeta function associated to the polynomial implies the existence of a fundamental solution. It is worth to mention that local zeta functions are very interesting on their own, for instance they serve in p-adic string theory to represent string amplitudes, see \cite{sym13060967} and the references therein. 
\end{itemize} 
However, apart from all the relevant applications, the existence of fundamental solutions for pseudo-differential operators is an interesting mathematical problem on its own, and we wish to contribute to it here by taking some elements of the theory to a noncommutative setting like $p$-adic Lie groups or noncommutative Vilenkin groups. In the process we will adapt to the $p$-adic case some ideas from the real case that have not been explored on totally disconnected groups, like the interactions between the representation theory of the group and the existence of solutions for an invariant operator, so, even when some of our methods are known techniques on real groups, their application to the unexplored setting of noncommutative Vilenkin groups is a novelty worth to highlight. 

\section{Preliminaries and statement of the main results}

In this section we provide some basic definitions and the main properties on locally compact Vilenkin groups that we will require. Our main goal here is to explain in detail our results and the way they extend to the noncommutative case the well known results about the existence of a fundamental solution for the Vladimirov-Taibleson operator and the Vladimirov Laplacian on $\Q_p^d$. The starting point is to recall the formal definition of the class of groups we are about to study: 
\begin{defi}\label{defilcvilenkingroup}\normalfont
We say that a topological group $G$ is a \emph{locally compact Vilenkin group} if $G$ is a  locally compact, Hausdorff, totally disconnected, unimodular, type I topological group, endowed with a strictly decreasing sequence of compact open subgroups $\mathscr{G}=\{G_n\}_{n \in \Z}$ such that \begin{enumerate}
    \item[(i)] It holds:$$| G_n /G_{n+1}|< \infty,$$ for every $n \in \Z.$
    \item[(ii)]$$G= \bigcup_{n \in \Z} G_n , \esp \esp \text{and} \esp \esp \bigcap_{n \in \Z} G_n = \{e\}.$$
    \item[(iii)]The sequence $\{G_n\}_{n \in \Z}$ form a basis of neighbourhoods at $e \in G$. If we also have $$\sup_{k \in \Z} |G_k/G_{k+1}|< \infty,$$we say that $G$ is a bounded-order Vilenkin group. If $|G_k/G_{k+1}|$ is constant we say that $G$ is a constant-order Vilenkin group.   
\end{enumerate}
\end{defi}
Vilenkin groups are locally profinite groups endowed with a sequence of compact open subgroups that forms a basis of neighbourhoods at the identity. The reason to consider such sequence is because by the Birkhoff-Kakutani theorem its existence guaranties the metrisability of these groups, and the obtained metric is compatible with the group structure. 

\begin{defi}\normalfont\label{metricvilenkingroups}
Let $G$ be a locally compact Vilenkin group with a sequence of compact open subgroups $\mathscr{G}= \{ G_n\}_{n \in \Z}$. For $x,y \in G$ the distance between $x$ and $y$ is defined as:
\[\varrho_\mathscr{G} (x,y) = |x y^{-1}|_{\mathscr{G}} :=\begin{cases} 0 & \esp \esp \text{if} \esp x=y, \\ |G_n|  & \esp \esp \text{if} \esp x y^{-1} \in G_n \setminus G_{n+1}.\end{cases}\] In this work we will consider exclusively constant-order Vilenkin groups, let us say $|G_k/G_{k+1}|=\varkappa$ for all $k \in \Z$, so that the metric on the group takes the form \[\varrho_\mathscr{G} (x,y) = |x y^{-1}|_{\mathscr{G}} :=\begin{cases} 0 & \esp \esp \text{if} \esp x=y, \\ \varkappa^{-n}  & \esp \esp \text{if} \esp x y^{-1} \in G_n \setminus G_{n+1}.\end{cases}\]
\end{defi}

Some examples of constant-order locally compact Vilenkin groups are:

\begin{exa}\label{exapadicnumbers}\normalfont
The field of $p$-adic numbers $G= \Q_p$, where $p$ is a prime number, together with the sequence of compact open subgroups $$G_n:=\{x \in \Q_p \esp : |x|_p \leq p^{-n}\}=p^n \Z_p.$$More generally, the additive group of any nonarchimedean local field $\K$, together with the compact open balls associated to the ultrametric topology of the field, are examples of locally compact abelian Vilenkin groups. The Vladimirov operator, which gives the analogue of fractional derivative in this setting, was originally studied in this framework and, for $\alpha \in \C$ with $\mathfrak{Re}(\alpha)>0$, it is defined as $$D^\alpha f (x):=\frac{1-p^\alpha}{1- p^{-\alpha - 1}} \int_{\Q_p} \frac{f(x-y) - f(x)}{|y|_p^{\alpha +1}} dy, $$acting initially on locally constant functions with compact support. Notice that for any $x,y \in G=\Q_p$, the distance function $| \cdot |_{\mathscr{G}}$ associated to $\mathscr{G}=\{G_n\}_{n \in \Z}$ has the property $|x-y|_{\mathscr{G}}= |x-y|_p$, so that we can write $$D^\alpha f (x):=\frac{1-p^\alpha}{1- p^{-\alpha - 1}} \int_{\Q_p} \frac{f(x-y) - f(x)}{|y|_{\mathscr{G}}^{\alpha +1}} dy.$$
\end{exa}

\begin{exa}\label{exapadicheisenberg}\normalfont
Let $p \neq 2$ be a prime number. We will denote by $\mathbb{H}_{d}(\Q_p)$, or simply $\mathbb{H}_d$ for short, the $2d+1$-dimensional Heisenberg group over $\Q_p$, in turn defined as \[
\mathbb{H}_{d}(\Q_p)= \left\{
  \begin{bmatrix}
    1 & x^t & z \\
    0 & I_{d} & y \\
    0 & 0 & 1 
  \end{bmatrix}\in GL_{d+2}(\Q_p) \esp : \esp x , y \in \Q_p^d, \esp z \in \Q_p \right\}. 
\]
$\mathbb{H}_{d}$ is an analytic $d$-dimensional manifold locally homeomorphic to $\Q_p^{2d+1}$ and the operations on $\mathbb{H}_{d}$ are analytic functions so $\mathbb{H}_{d}$ is a $p$-adic Lie group. Let us denote by $\mathfrak{h}_{d}$ its associated Lie algebra. We can write explicitly \[
\mathfrak{h}_{d}(\Q_p)= \left\{
  \begin{bmatrix}
    0 & \textbf{a}^t & c \\
    0 & 0_{d} & \textbf{b} \\
    0 & 0 & 0 
  \end{bmatrix}\in \mathcal{M}_{d+2}(\Q_p) \esp : \esp \textbf{a} , \textbf{b} \in \Q_p^d, c \in \Q_p \right\}, 
\]which is a nilpotent Lie algebra over $\Q_p$. This implies that the global convergence of the exponential map, and for any $u \in \mathfrak{h}_d$, let us say \[u:= \begin{bmatrix}
    0 & \textbf{a}^t & c \\
    0 & 0_{d} & \textbf{b} \\
    0 & 0 & 0 
  \end{bmatrix},\]the exponential map evaluates to \[\mathbb{exp} (u) = \begin{bmatrix}
    1 & \textbf{a}^t & c + \frac{1}{2} \textbf{a} \cdot \textbf{b} \\
    0 & I_{d} & \textbf{b} \\
    0 & 0 & 1 
  \end{bmatrix}.\]The exponential map transform sub-ideals of the Lie algebra $\mathfrak{h}_{d} $ to subgroups of $\mathbb{H}_{d}$. Actually, we can turn the exponential map into a group homomorphism by using the Baker–Campbell–Hausdorff formula. Let us define the operation ``$\diamond$" on $\mathfrak{h}_{d} $ by $$X \diamond Y:= X + Y + \frac{1}{2} [X,Y].$$ Then, clearly $ ( \mathbb{H}_d (\Q_p), \times) \cong (\mathfrak{h}_{d} (\Q_p) , \diamond )$, where "$\times$" is the matrix product, is a locally profinite topological group, and it can be endowed with the sequence of subgroups $G_n := ( \mathbb{exp}( p^n \mathfrak{h}_{d}(\Z_p)), \times) \cong (p^n \mathfrak{h}_{d}(\Z_p) , \diamond)$, where $$p^n \mathfrak{h}_{d}(\Z_p)= p^n \Z_p X_1 +...+ p^n \Z_p X_d + p^n \Z_p Y_1 + ...+ p^n\Z_p Y_{d} + p^{2n} \Z_p Z,$$so $(\mathfrak{h}_{d} (\Q_p) , *) \cong \mathbb{H}_d (\Q_p)$ is a locally compact Vilenkin group. Here $X_1,...,Y_d,Z$ are the usual $\Q_p$-basis for $\mathfrak{h}_d$. An alternative  definition for $\mathbb{H}_d (\Q_p)$ can be given in terms of symplectic vector spaces. Let $V$ be a $2d$-dimensional vector space over the field $F=\Q_p$ so that $V \cong \mathbb{Q}_p^{2d}$. Let $\langle , \rangle_V$ be a non-degenerated symplectic form. The Heisenberg group associated to $V$ is the group with underlying set $V\times F$ and the following group operation: $$(x,y,z) \cdot (x',y',z') := (x+ x' , y + y' , z + z' + \langle (x,y) , (x',y' ) \rangle , \esp \esp \esp (x,y), (x',y') \in V\cong \Q_p^d \times \Q_p^d.$$One can see that the previous definition and the initial definition are equivalent in the sense they define isomorphic groups.    
\end{exa}

\begin{exa}\label{exapadicnilpotentgroups}\normalfont
More generally, let $p$ be a prime number, $p\neq 2$, and let $\K$ be nonarchimedean local field with ring of integers $\mathcal{O}_\K$, prime ideal $\p = \mathscr{p} \mathcal{O}_\K$ and residue field $\mathcal{O}_\K / \p \cong \mathbb{F}_q$, where $q= p^{[\K : \Q_p]}$. Let $\mathfrak{g}$ be a nilpotent $\K$-Lie algebra. We say that $\mathfrak{g}$ is \emph{graded} if there exists a decomposition of the form $$\mathfrak{g} = \bigoplus_{k=1}^{r'} V_{k}, \esp \esp \esp [V_i , V_j]\subseteq V_{i+j}.$$Let $V_{\nu_1},...,V_{\nu_r}$ the non-trivial subspaces appearing in the above decomposition, so that $$\mathfrak{g} = \bigoplus_{k=1}^{r} V_{\nu_k}, \esp \esp b_k := dim(V_{\nu_k}).$$We will call the number $Q:=\nu_1 b_1 + \nu_2 b_2 + ...+ \nu_r b_r$ the homogeneous dimension of $\mathfrak{g}$. In particular, if the elements of the first subspace appearing in the decomposition of $\mathfrak{g}$ generate the whole algebra, we say that $\mathfrak{g}$ is \emph{stratified}. If the Lie algebra of a nilpotent Lie group $G$ is graded, we say that $G$ is a \emph{graded Lie group}. When the Lie algebra is stratified we say that $G$ is a \emph{stratified group}.
A consequence of the graded structure on a $\K$-Lie algebra is the existence of dilations on the group. That is, there exists a family of linear mappings $D_\gamma : \mathfrak{g} \to \mathfrak{g}$, $\gamma \in \K^*$, satisfying :
\begin{itemize}
        \item[-]the mappings $D_\gamma$ are diagonalisable and each $V_j$ is the eigenspace associated to one of the eigenvalues of $D_\gamma$.
        \item[-] each $D_\gamma$ is a morphism of the Lie algebra $\mathfrak{g}$, that is, a linear mapping from $\mathfrak{g}$ to itself which respects the Lie bracket: $$[D_\gamma X , D_\gamma Y] = D_\gamma [X, Y], \esp \esp \text{for all} \esp X, Y \in \mathfrak{g}, \esp \esp \text{and all} \esp \gamma \in \K^*.$$
\end{itemize}
The dilations are transported to the Lie group by the exponential map in the following way: the maps $$\mathbb{exp} \esp  \circ D_\gamma \circ \esp  \mathbb{exp}^{-1}, \esp \esp \gamma \in \K^*  ,$$are automorphisms of the group $G$. We will also denote by $D_\gamma$ the dilations on the group. We may write $$\gamma x := D_\gamma (x).$$
                     
A graded group $G$ can be endowed with the structure of a locally compact Vilenkin group with a sequence compatible with the dilations. For doing so let $X_1 , ...,X_d$ be a $\K$-basis of eigenvectors in $\mathfrak{g}$ associated to the dilations $D_\gamma$ so $D_\gamma$ is diagonal in this basis. Define $$\mathfrak{g}_0 := \mathcal{O}_\K X_1 + ... + \mathcal{O}_\K X_d, \esp \esp \text{and} \esp \esp G_0 := \mathbb{exp}(\mathcal{O}_\K X_1 + ... + \mathcal{O}_\K X_d). $$Hence, if $|\gamma|_\K = q^{-n}$, $$G_n := \mathbb{exp}( \gamma \mathfrak{g}_0) = D_\gamma ( G_0)  , \esp \esp  \gamma \mathfrak{g}_0 =  \mathscr{p}^{n \nu_1 } \mathcal{O}_\K X_1 + ... + \mathscr{p}^{n \nu_r} \mathcal{O}_\K X_d = D_{\gamma} (\mathfrak{g}_0),$$defines a sequence $\{ G_n \}_{n \in \Z}$ of compact open sub-groups of $G$ such that $$G= \bigcup_{n \in \Z} G_n , \esp \esp \esp \text{and} \esp \esp \esp \bigcap_{n \in \Z} G_n = \{ e \}.$$
With the above filtration $G$ is a constant Vilenkin group since \begin{align*}
    | G_n / G_{n+1}| & = | \mathfrak{g}_n / \mathfrak{g}_{n+1}| \\ & = q^{\nu_1 b_1 + ... + \nu_r b_r} \\ &= q^{Q},
\end{align*} for all $n \in \Z$. For instance in the case $\K = \Q_p$ and $G= \mathbb{H}_d (\Q_p)$ we can define the dilations $$D_\gamma (x, y ,z) := \mathbb{exp}( \gamma x_1 X_1 + ... \gamma x_d X_d + \gamma y_1 Y_1 + ...+ \gamma y_d Y_d + \gamma^2 z Z) = (\gamma x, \gamma y , \gamma^2 z),$$and the sequence of subgroups associated to this dilations, the same defined in Example \ref{exapadicheisenberg}, make $\mathbb{H}_d (\Q_p)$ a constant-order Vilenkin group since $$| G_n / G_{n+1}|=p^{2d+2}, \esp \esp \text{for all}\esp \esp n \in \Z.$$
\end{exa}

\begin{exa}\label{exaEngelalgebra}\normalfont
Consider the $4$-dimensional $p$-adic Lie group $\mathbb{E}_4 (\Q_p)$, or $\mathbb{E}_4$ for simplicity, which is defined here as the exponential image of the filiform $\Q_p$-Lie algebra $\mathfrak{e}_4$, called by some authors the Engel algebra, defined in terms of the $\Q_p$-basis $\{X , Y_1 , Y_2 , Y_3\}$ and the commutation relations $$[X, Y_i] = Y_{i+1}, \esp \esp i=1,2, \esp \esp \text{and} \esp \esp \esp \esp [X,Y_3] =0.$$We can think on $\mathfrak{e}_4$ as the matrix algebra containing all the matrices of the form  \[ \begin{bmatrix}
    0 & x & 0 & y_3 \\
    0 & 0 & x &  y_2  \\
    0 & 0 & 0 & y_1 \\ 
    0 & 0 & 0 & 0
  \end{bmatrix}, \esp \esp \esp x , y_1 , y_2 , y_3 \in \Q_p.\]In this way $\mathbb{E}_4$, which is the exponential image of the $\mathfrak{e}_4$, is a nilpotent subgroup of $GL_4 (\Q_p)$ that we call here \emph{the Engel group}. With the coordinates $$(x,y_1,y_2,y_3):= \mathbb{exp}( xX +y_1 Y_1 + y_2 Y_2 + y_3Y_3),$$we can identify $\mathbb{E}_4$ with $\Q_p^4$ endowed with the group law: \begin{align*}
      (x,y_1,y_2 , y_3) \times (x',y_1',y_2' , y_3'):= (x + x',y_1 + y_1',y_2 +y_2' -x y_1',y_3 +  y_3' + \frac{1}{2} x^2 y_1' -x y_2').
  \end{align*}Denote by "$\diamond$" the BCH product on $\mathfrak{e}_4$. Then $\mathbb{E}_4 \cong (\mathfrak{e}_4, \diamond)$ is a constant-order locally compact Vilenkin group since it can be endowed with the sequence of compact open subgroups $(G_n , \times) \cong (p^n \mathfrak{e}_4 , \diamond) $ where $$G_n:= \mathbb{exp}( p^n \mathfrak{e}_4) := \mathbb{exp}( p^n \Z_p X + p^n \Z_p Y_1 + p^{2n} \Z_p Y_2 + p^{3n} \Z_p Y_3) ,$$and it clearly holds $|G_n / G_{n+1}| = p^7$ for all $n \in \Z$. Here $Q=7$ is the homogeneous dimension of $\mathbb{E}_4$.   
\end{exa}
There are two ways to generalise the Vladimirov operator defined in Example \ref{exapadicnumbers} to the multidimensional case. The first one is the Taibleson operator, or the Vladimirov-Taibleson operator, a very special pseudo-differential operator with many works about its properties in the mathematical literature. It is defined via the formula $$\mathscr{D}^\alpha f (x):=\frac{1-p^\alpha}{1- p^{-\alpha - d}} \int_{\Q_p^d} \frac{f(x-y) - f(x)}{\|y\|_p^{\alpha +d}} dy, \esp \esp \esp \mathfrak{Re}(\alpha)>0,$$or equivalently via Fourier transform as $$\mathscr{D}^\alpha f (x):= \int_{\Q_p^d} \| \xi \|_p^\alpha \widehat{f}(\xi) e^{2 \pi i \{ x \cdot \xi \}_p} d \xi, $$ acting initially, as before, on locally constant functions with compact support. This operator is an invariant operator on $\Q_p^d$ whose convolution kernel, called the \emph{mulidimensional Riesz kernel}, is given by the expression $$\langle \mathfrak{r}_\alpha , f \rangle = \frac{1-p^{-n}}{1-p^{\alpha -n}} f(0) + \frac{1-p^{-\alpha}}{1-p^{\alpha-n}}\int_{\|x \|_p>1} \| x\|_p^{\alpha - n} f(x) dx + \frac{1-p^{-\alpha }}{1-p^{\alpha - n}} \int_{\| x\|_p \leq 1}\| x\|_p^{\alpha - n} (f(x)-f(0)) dx,$$so that for $\mathfrak{Re}(\alpha) >0$ and any locally constant function $f$ with compact support it holds $\mathscr{D}^\alpha f = f * \mathfrak{r}_{-\alpha}$. See \cite{RodriguezJ} for more details. Here we are using the notation $$\|x\|_p=\max_{1 \leq k \leq   d} |x_i|_p.$$The heat equation with the Vladimirov-Taibleson operator, or diffusion equation with $\mathscr{D}^\alpha$, is the parabolic equation \begin{equation}\label{heateq}
    \frac{d}{dt} u(t,x) + \mathscr{D}^\alpha u (t,x)=0, \esp \esp t>0, \esp \esp x \in \Q_p^d,
\end{equation}and its applications to mathemathical physics and other areas has been widely discussed in the literature. For example the reader can consult \cite{avetisov} where the connections between this and more general convolution operators with ultrametric diffusion models are explored. From the mathematical point of view, Equation (\ref{heateq}) is interesting on its own and its fundamental solution, usually called \emph{the heat kernel} $h_\alpha (t,x)$ associated to $\mathscr{D^\alpha}$, has some remarkable properties, for instance it is a positive function that defines a transition density of a right continuous strict Markov process without second kind discontinuities \cite{RodriguezJ}. Another remarkable property of this kernel is the estimate $$h_\alpha (t,x) \asymp \frac{t}{(t^{1/\alpha} + \| x \|_p)^{\alpha+d}}, \esp \esp \esp \text{for} \esp \alpha >0,$$which in principle guaranties that for $ 0< \alpha < d$ we can obtain the fundamental solution of $\mathscr{D}^\alpha$ by integrating $h_\alpha (t,x)$ in the variable $t$, as it is proven in \cite{IsoMarkovSemi}. Anyway, there are more general methods to obtain explicitly the fundamental solution for $\mathscr{D}^\alpha$, which is actually given by \[E_\alpha (x) = \begin{cases}
\frac{1 - p^{-\alpha}}{1 - p^{\alpha -d }} \| x\|_p^{\alpha -d}, & \esp \esp \text{if} \esp \mathfrak{Re}(\alpha) \neq d, \\ \frac{1 - p^n}{p^n \ln{ p } } \ln{\|x \|_p}, & \esp \esp \text{if} \esp \mathfrak{Re}(\alpha)=d ,
\end{cases}
\]but it is important to remark that for the case $\mathfrak{Re}( \alpha) \geq d$ one obtains in principle just a \emph{formal fundamental solution}, meaning that the linear functional $E_\alpha$ is not well defined for every locally constant function with compact support. The reason is that via Fourier transform $\mathcal{F}_{\Q_p^d} E_\alpha (\xi) = \| \xi \|_p^{-\alpha}$ but the function $\| \xi \|_p^{-\alpha}$ is not locally integrable for $\mathfrak{Re}(\alpha) \geq d$. See \cite{RodriguezJ, fundsoltaibleson} for more details about the Taibleson operator and its fundamental solution.

There is a natural way to generalise the above ideas to constant-order Vilenkin groups. Consider $G= \Q_p^d$ together with the sequence of compact open subgroups $$G_n:=\{x \in \Q_p \esp : \|x\|_p \leq p^{-n}\}=p^n \Z_p^d.$$ Then $\Q_p^d$ is a constant-order Vilenkin group and for the metric associated to $\mathscr{G}=\{G_n\}_{n \in \Z}$ it holds $|x-y|_{\mathscr{G}}= \|x-y\|_p^d$. Accordingly, we can write the Taibleson operator in terms of the ultrametric $|\cdot|_{\mathscr{G}} $ as:
$$\mathscr{D}^\alpha f (x):=\frac{1-p^{d(\alpha/d)}}{1- p^{-d(\alpha/d + 1)}} \int_{\Q_p^d} \frac{f(x-y) - f(x)}{|y|_{\mathscr{G}}^{\alpha/d +1}} dx.$$ So actually, the Taibleson operator is a particular example of a class of operators on ultrametric spaces defined in terms a certain ultrametric function \cite{IsoMarkovSemi} that, in particular for Vilenkin groups, we choose as the ultrametric associated to a sequence of compact open subgroups.  Hence we can "extend" the study of the Vladimirov-Taibleson operator, in principle, to constant order Vilenkin groups simply by replacing in the original definition the $p$-adic norm with the natural ultrametric of the group.  
\begin{defi}\label{defivtoperator}
Let $G$ be a constant-order Vilenkin group, endowed with the sequence of compact open subgroups $\mathscr{G} = \{G_n\}_{n \in \Z}$, let us write $|G_n/G_{n+1}|=\varkappa$ for all $n \in \Z$. We define for $\mathfrak{Re}(\alpha)>0$ the Vladimirov-Taibleson operator $D^\alpha$, acting initially on locally constant functions with compact support, via the formula:
$$D^\alpha f (x) := \frac{1-\varkappa^{\alpha}}{1 -\varkappa^{-\alpha -1}}  \int_{G} \frac{f(xy^{-1})-f(x)}{|y|_{\mathscr{G}}
^{-\alpha - 1}}dx.$$
\end{defi}

One of our contributions in this paper is to compute the fundamental solution for the above defined operator on constant-order noncommutative Vilenkin groups, extending in this way to the noncommutative case the well known results for the Vladimirov-Taibleson operator on local fields and abelian Vilenkin groups. 

\begin{teo}\label{teofundamentalsolVTopLC}
A fundamental solution for the Vladimirov-Taibleson operator $D^\alpha$, $\mathfrak{Re}(\alpha)>0 , \mathfrak{Im}(\alpha) \notin (2 \pi i / \ln \varkappa) \Z$, $\mathfrak{Re}(\alpha) \neq 1$, is given by $$E_\alpha (x) =  \frac{1 - \varkappa^{-\alpha}}{1 - \varkappa^{\alpha -1 }} | x|_\mathscr{G}^{\alpha -1} .$$
\end{teo} 

Observe that the given definition of the Vladimirov-Taibleson operator and its fundamental solution depend on the choice of ultrametric on the group, so that the definition depends on the choice a sequence of compact open subgroups $\mathscr{G}$, implying that there can be several definitions of this operator for the same group. However in some cases there is a natural choice of ultrametric, like for $\Q_p^d$ where the definition is given in terms of the $p$-adic norm, and more generally the same occurs for graded $\K$-Lie groups where there is a natural and convenient definition of Vladimirov-Taibleson operator in terms of an homogeneous ultrametric function. A modification we need to do with respect to the definition previously given is that for graded $\K$-Lie groups we want to work with an specific choice of distance function on the group. The reason for this choice is that we want our distance function to be an homogeneous quasi-norm in the sense of the following definition:
\begin{defi}\label{defihomoquasinorm}\normalfont
Let $G$ be a graded $\K$-Lie group together with the dilations $D_\gamma, \gamma \in \K^*$. An \emph{homogeneous quasi-norm} is a continous non-negative function $| \cdot| : G \to [0 , \infty)$ satisfying 
\begin{itemize}
    \item $|x^{-1}| = |x|$ for all $x \in G$,
    \item $|\gamma x| = | \gamma |_\K |x| $ for all $x \in G$ and $\gamma \in \K^*$,
    \item $|x|=0$ if and only if $x = e$.
\end{itemize}
\end{defi}
\begin{defi}\label{defimetricgradedgroups}\normalfont
Let $G$ be a graded $\K$-Lie group with homogeneous dimension $Q$, considered as a Vilenkin group with the sequence of subgroups $\mathscr{G}=\{ G_n\}_{n \in \Z}$ determined by the dilations on the algebra. We define the distance function $| \cdot |_G$ on $G$ by  \[ = |x y^{-1}|_{G} :=\begin{cases} 0 & \esp \esp \text{if} \esp x=y, \\ |x y^{-1}|_\mathscr{G}^{1/Q}  & \esp \esp \text{if} \esp x y^{-1} \in G_n \setminus G_{n+1}.\end{cases}\]
\end{defi}
\begin{rem}
The above defined metric is an homogeneous quasi-norm in the sense of Definition \ref{defihomoquasinorm}. To see this take $xy^{-1} \in G_n \setminus G_{n+1}$ and $\gamma \in \K^*$ with $|\gamma|_\K = q^{-k}$. Thus $\gamma(xy^{-1}) \in G_{n+k} \setminus G_{n+k+1}$ so $|\gamma(xy^{-1})|_{\mathscr{G}}=q^{-Q(n+k)} = |\gamma|_\K^{Q} |xy^{-1}|_\mathscr{G}$.
\end{rem}

Considering the above, the definition of Vladimirov-Taibleson operator that we will use on graded $\K$-Lie groups is the following:

\begin{defi}\label{defiVToperatorGradedlie}\normalfont
Let $G$ be a graded $\K$-Lie group with homogeneous dimension $Q$. Take $\alpha \in \C$ with $\mathfrak{Re}(\alpha)>0$. We define the Vladimirov-Taibleson operator on $G$ as the linear operator acting initially on locally constant functions with compact support according to the following formula: $$\mathscr{D}^\alpha f(x) := \frac{1 - q^\alpha}{1 - q^{- (\alpha + Q)}} \int_G \frac{f(xy^{-1}) - f(x)}{|y|_G^{ \alpha + Q}} dy, \esp \esp x \in G.$$
\end{defi}
And for this operator its fundamental solution is:
\begin{teo}\label{fundamentalsolutiongradedliegroups}
A fundamental solution for $\mathscr{D}^\alpha$, $\mathfrak{Re}(\alpha)>0 , \mathfrak{Im}(\alpha) \notin (2 \pi i / \ln q) \Z$, $\mathfrak{Re}(\alpha) \neq Q$ is given by $$E_\alpha (x) =  \frac{1 - q^{-\alpha}}{1 - q^{\alpha -Q }} | x|_G^{\alpha -Q}  \esp \esp  \esp \mathfrak{Re}(\alpha) \neq Q.$$
\end{teo}
\begin{rem}
The Vladimirov-Taibleson $\mathscr{D}^\alpha$ is an $\alpha$-homogeneous operator since \begin{align*}
    \mathscr{D}^\alpha (f \circ D_\gamma) (x) &= \frac{1 - q^\alpha}{1 - q^{- (\alpha + Q)}} \int_G \frac{f(\gamma(xy^{-1})) - f(\gamma x)}{|y|_G^{ \alpha + Q}} dy \\ &= \frac{1 - q^\alpha}{1 - q^{- (\alpha + Q)}} |\gamma|_\K^{\alpha + Q} \int_G \frac{f(\gamma(xy^{-1})) - f(\gamma x)}{|\gamma y|_G^{ \alpha + Q}} dy \\ &=\frac{1 - q^\alpha}{1 - q^{- (\alpha + Q)}} |\gamma|_\K^{\alpha } \int_G \frac{f((\gamma x)y^{-1}) - f( \gamma x)}{| y|_G^{ \alpha + Q}} dy   \\&= |\gamma|_\K^\alpha (\mathscr{D}^\alpha f) \circ D_\gamma (x). 
\end{align*}
We will come back to this fact and some of its implications on Section 2 where we study homogeneous operators like the "Vladimirov Laplacian" and the "Vladimirov Sub-Laplacian"  on the $p$-adic Heisenberg group.  
\end{rem}

\begin{rem}
 When $0<\alpha<Q$, there is an alternative way to obtain the heat kernel, its estimates, and the fundamental solution of the Vladimirov-Taibleson operator on graded $\K$-Lie groups, and more generally on constant-order Vilenkin groups, by using isotropic Markov semigroups and non-local Dirichlet forms, see \cite{DirichletForms, IsoMarkovSemi} for all the details. In those works the authors studied Markov semigroups on general ultrametric measure spaces and obtained some interesting estimates for the heat kernel. We want to remark that the same estimates hold true in particular for constant-order Vilenkin groups and graded $\K$-Lie groups, even thought the authors in \cite{DirichletForms, IsoMarkovSemi} do not point out this fact. We will give more details about it in Section 1.5. However, our approach is different in nature as it involves the representation theory of Vilenkin groups, and it also produces a more general result since we can get the fundamental solution for $\alpha \in \C$ with $Q \neq \mathfrak{Re}(\alpha) >0$, and even for $\mathfrak{Re}(\alpha)=Q$ when the group is compact.
\end{rem}
The second operator that generalise the Vladimirov operator to the multidimensional case is the \emph{Vladimirov Laplacian} $\MO^\alpha$, $\alpha>0$, introduced by Vladimirov in \cite{Vladimirov_1988} for the particular case $\alpha=2$. We define it here as the invariant $\alpha$-homogeneous operator acting on locally constant functions with compact support via the formula: $$ \MO^\alpha f(x) := \sum_{j=1}^d \partial^\alpha_{x_j} f(x),$$where $$\partial^\alpha_{x_j} f(x) : = \frac{1 - p^\alpha}{1 - p^{- (\alpha + 1)}}\int_{\Q_p} \frac{ f(x-x_j) - f(x)}{|x_j|_p^{\alpha + 1}} dx_j. $$The operator $\MO^\alpha$ was studied in \cite{IsoMarkovSemi} and the existence of its fundamental solution was proven for $ 0 < \alpha < d$. In addition, the following properties of $\MO^\alpha$ were established: 

\begin{itemize}
\item $(\MO^\alpha , \mathcal{D}(\Q_p^d))$ is a non-negative symmetric operator, where $\mathcal{D}(\Q_p^d)$ denote the collection of locally constant functions with compact support on $\Q_p^d$. 
\item $(\MO^\alpha , \mathcal{D}(\Q_p^d))$ admits a system of compactly supported eigenfunctions. Besides the operator is essentially self adjoint.
\item The semigroup $e^{-t \MO^\alpha}$ is symmeric and Markovian. It admits a heat kernel $h_{\MO^\alpha} $ that satisfies the following estimate  $$h_{\MO^\alpha_2} ( t ,x) \asymp t^{-d/\alpha} \prod_{j=1}^d \min\{1 , \frac{ t^{1 + 1/\alpha}}{ |x_j|_p^{1+\alpha}} \} .$$
\item The semigroup is transient if and only if $0 < \alpha< d$.
\end{itemize}
For more general graded $\K$-Lie groups the question we want to investigate is whether something similar occurs for an operator like $$\MO^\alpha :=\sum_{k=1}^r \sum_{j=1}^{b_k} \partial_{X_{k,j}}^{\alpha/\nu_k} ,$$where $\{ X_{k,j} \}$ is a Malcev basis associated to the gradation $$\mathfrak{g} = \bigoplus_{k=1}^{r} V_{\nu_k}, \esp \esp b_k := dim(V_{\nu_k}).$$Here we are using the notation introduced in the following definition:
\begin{defi}\normalfont
Let $G$ be a graded $\K$-Lie group with Lie algebra $\mathfrak{g}$. We define the \emph{directional VT operator} in the direction of $X \in \mathfrak{g}$ by the formula $$\partial_X^\alpha f (x) :=  \frac{1 - q^\alpha}{1 - q^{- (\alpha + 1)}} \int_{\K} \frac{f(x \cdot \mathbb{exp}(tX)^{-1}) - f(x)}{|t |_\K^{\alpha + 1}} dt .$$
\end{defi}

We call $\MO^\alpha$ the \emph{Vladimirov Laplacian} on $G$, and in particular for the Heisenberg group $\mathbb{H}_d$ it takes the form: $$\MO^\alpha = \sum_{j=1}^{d} \partial_{X_j}^\alpha + \partial_{Y_j}^\alpha + \partial_Z^{\alpha/2},$$and for the Engel group $\mathbb{E}_4$: $$\MO^\alpha = \partial_{X}^\alpha + \partial_{Y_1}^\alpha + \partial_{Y_2}^{\alpha/2} + \partial_{Y_3}^{\alpha/3}.$$ We will show the existence of a fundamental solution for $\MO^\alpha$ on $\mathbb{H}_d$ and $\mathbb{E}_4$ using the group Fourier analysis, and we will also obtain an estimate on their associated heat kernels as an application of the results in \cite{IsoMarkovSemi}. We collect our results in the following Theorem:
\begin{teo}\label{teoheatkernelpropertiesgroups}
Let $G$ be either $\mathbb{H}_d$ or $\mathbb{E}_4$. Denote by $\MO^\alpha_2$ the self-adjoint extension of $\MO^\alpha$ to $L^2 (G)$. Then the heat kernel $h_{\MO^\alpha_2} $ associated to the Vladimirov Laplacian $\MO^\alpha$ on $G$ has the following properties: 

\begin{itemize}
    \item[(i)] $h_{\MO^\alpha_2} (t, \cdot) *  h_{\MO^\alpha_2} (s, \cdot) = h_{\MO^\alpha_2} (t + s , \cdot )  $, for any $s,t>0$.
    \item[(ii)] $h_{\MO^\alpha_2} ( | \gamma |_p^\alpha t, D_\gamma (x)) = |\gamma|^{-Q}_p h_{\MO^\alpha_2} ( t , x)$, for all $x \in G$ and any $t>0$, $\gamma \in \Q_p^*$.
\item[(iii)] $h_{\MO^\alpha_2} ( t , x) = \overline{h_{\MO^\alpha_2} ( t , x^{-1})},$ for all $x \in G$.  
\item[(iv)] The heat semigroup $e^{- t \MO_2^\alpha}$ is symmetric and Markovian. Moreover, the following estimate holds for its kernel: $$h_{\MO^\alpha_2} ( t , x) \asymp t^{-Q/\alpha}.$$
\end{itemize}
\end{teo}

Theorem \ref{teoheatkernelpropertiesgroups}, among other things, guaranties that for $0<\alpha<Q$,  the fundamental solution and the Riesz potential of the operators $\MO^\alpha$ is well defined.
\begin{coro}
Let $G$ be either $\mathbb{H}_d$ or $\mathbb{E}_4$. Let $0 < \alpha < Q.$ Then a fundamental solution for the Vladimirov Laplacian exists and it defines an $ \alpha - Q$-homogeneous function given by $$\textbf{h}_{\MO^\alpha_2}(x) := \int_0^{\infty} h_{\MO^\alpha_2} ( t, x) dt.$$
Consequently: $$\textbf{h}_{\MO^\alpha_2}(x ) \asymp |x|_G^{\alpha - Q}.$$ 
\end{coro}
\begin{coro}
Let $\beta\in \C$ such that $0< \mathfrak{Re}(\beta) < Q$. Then the linear operator $L^\beta:=(\MO_2^\alpha)^{\beta/\alpha}$ defined via functional calculus possesses a fundamental solution, which is an $\beta - Q$ homogeneous distribution, determined by the Riesz potential $$\mathcal{I}_\beta (x) := \frac{1}{\Gamma(\beta/\alpha)} \int_0^\infty t^{\beta/\alpha - 1} h_{\MO^\alpha_2} ( t, x)dt.$$Furthermore, for $0<\beta<Q$ we have $$\mathcal{I}_\beta (x) \asymp |x|_G^{\beta - Q}.$$  
\end{coro}
We will explain in detail the properties of the heat kernel and the existence of a fundamental solution for $\MO^\alpha$ on $\mathbb{H}_d$ and $\mathbb{E}_4$ in Section 3. Finally, we will discuss how our ideas can be applied to more general graded groups and also to more general homogeneous operators.   

\section{The Vladimirov-Taibleson operator}
The proofs of Theorems \ref{teofundamentalsolVTopLC} and \ref{fundamentalsolutiongradedliegroups} consist of two parts. The first one is finding the fundamental solution of the Vladimirov-Taibleson on \emph{compact Vilenkin groups}, which we define formally in Definition \ref{defivilenkingroupscompact}. For the sake of our discussion, all the compact Vilenkin groups that we are going to consider are subgroups of a bigger locally compact Vilenkin group. Actually, if $G$ is a locally compact Vilenkin group endowed with the sequence of compact open subgroups $\mathscr{G}=\{G_n\}_{n \in \Z}$, we are interested exclusively on the subgroups $G_k$ in this sequence, which are compact Vilenkin groups on their own, together with the sequence of compact open subgroups $\mathscr{G}_k=\{G_{k+n}\}_{n \in \N_0}$. We will compute the fundamental solution of the Vladimirov-Taibleson operator on these groups and after that, for the second part of the proof, we will use the fundamental solution obtained in the compact case to obtain the fundamental solution in the non-compact case. 
\subsection{Compact Vilenkin groups}
  
  \begin{defi}\normalfont\label{defivilenkingroupscompact}
We will say that a topological group $G$ is a compact noncommutative Vilenkin group if $G$ is a profinite group that possesses a decreasing sequence of compact open normal subgroups $G:= G_0 \supset G_1 \supset G_2 \supset...$ such that \begin{itemize}
    \item[(i)] $2 \leq \kappa_n := | G_n / G_{n+1} | < \infty$, for all $ n \in \N_0. $
    \item[(ii)] $\{ G_n \}_{n \in \N_0}$ is a basis of neighbourhoods at $e \in G$ and $\bigcap_{n \in \N_0} G_n = \{e\}.$ 
\end{itemize}
\end{defi}

Let $G$ be a compact Vilenkin group together with the sequence of compact open subgroups $\{G_n\}_{n \in \N_0}$. We will denote by $Rep(G)$ the collection of all continuous finite-dimensional unitary representations on $G$, and $\widehat{G}$ will denote the collection of irreducible representation in $Rep(G)$. By the Peter-Weyl Theorem the matrix entries of the representations $[\xi] \in \widehat{G}$ form an orthonormal basis of the space $L^2(G)$ and therefore any function $f \in L^2 (G)$ can be written as $$f(x) =\mathcal{F}_G^{-1} \circ \mathcal{F}_G (f)(x) = \mathcal{F}_G^{-1}(\widehat{f})(x)= \sum_{[\xi] \in \widehat{G} } d_\xi Tr(\xi(x) \widehat{f} (\xi)),$$where $\widehat{f}(\xi)$ denotes the Fourier coefficient of $f$ with respect to $[\xi] \in \widehat{G}$ defined as $$\widehat{f}(\xi) = \mathcal{F}_G [f] (\xi) := \int_G f(x) \xi^* (x) dx, $$ $dx$ denotes the unique normalised Haar measure on $G$, and the inverse Fourier transform is defined as $$\mathcal{F}_G^{-1} \varphi (x) := \sum_{[\xi] \in \widehat{G}} d_\xi Tr[\xi(x) \varphi (\xi)].$$

Recall that every matrix representation $\pi:G \to GL_{d_\pi} (\C)$, $[\pi] \in Rep(G)$, is a continuous function and $GL_{d_\pi}(\C)$ does not contain small subgroups. Thus we can find a neighbourhood $V$ of $I_{d_\pi}$ such that the only subgroup of $GL_{d_\pi} (\C)$ contained in $V$ is $\{ I_{d_\pi} \}$. Let $U:= \pi^{-1} (V)$. Clearly $U$ is a neighbourhood of $e \in G$ and hence some $G_n$ is contained in $U$ and $\pi(U) \subset \{I_{d_\pi}\}$. This shows that the kernel of every $\pi$ with $[\pi] \in Rep(G)$ must contain one of the sub-groups $G_n$. We will denote $$n_\pi := \min \{n \in \N_0 : \pi |_{G_n} = I_{d_\pi} \},$$for an arbitrary $[\pi] \in Rep(G)$. We also use the notation 
$$G_n^\bot := \{ [\pi] \in Rep(G) \esp : \esp \pi |_{G_n} = I_{d_\pi}\}, \esp \esp Rep_n(G):= G^\bot_n \setminus G_{n-1}^\bot, \esp \esp \widehat{G}_n := Rep_n (G) \cap \widehat{G}.$$ 
We have just proven the decomposition of $Rep(G)$ and $\widehat{G}$ as the disjoint unions $$Rep(G)=\bigcup_{n \in \N_0} Rep_n(G) , \esp \esp \text{and} \esp \esp \widehat{G} =\bigcup_{n \in \N_0} \widehat{G}_n,$$and in this way any $f \in L^2(G)$ can be written as $$f(x)= \sum_{n \in \N_0} \sum_{[\xi] \in \widehat{G}_n} d_\xi Tr(\xi(x) \widehat{f}(\xi)),$$ and the action of a densely defined linear operator $T$ on any function $f \in L^2 (G)$ as $$T f (x) = \sum_{n \in \N_0} \sum_{[\xi] \in \widehat{G}_n} d_\xi Tr \big( \xi (x) \sigma_T (x , \xi) \widehat{f} (\xi) \big).$$Here the symbol of the operator is defined as $$\sigma_T (x , \xi):=\xi^* (x) T \xi (x),$$and in particular it does not depend on the variable $x \in G$ when $T$ is left invariant.  
\begin{defi}\label{weigth}

We define the function $\langle \cdot\rangle_{\mathscr{G}} : Rep(G) \to \R$ in the following way: 
\[ \langle \pi \rangle_{\mathscr{G}} := \begin{cases}
1 \esp & \esp \text{if} \esp \esp \pi \esp \text{is the identity representation;} \\
|G/G_n| \esp & \esp \text{if} \esp \esp \pi \in Rep_n (G), \esp \esp n \in \N_0.  
\end{cases}
\]In this way it holds for any pair of nontrivial representations $$\langle \pi \otimes \xi \rangle_{\mathscr{G}} = \max\{\langle \xi \rangle_{\mathscr{G}} , \langle \pi \rangle_{\mathscr{G}} \}, \esp \esp [\pi], [\xi] \in Rep (G).$$
\end{defi}

In the context of compact Vilenkin groups we only consider densely defined operators whose domain contains one of the two relevant classes of "smooth functions" on $G$. The first class of relevant smooth functions on a compact Vilenkin group is the collection of locally constant functions on $G$ with a fixed index of locally constancy, here denoted by $\mathcal{D}(G)$. That is, $f \in \mathcal{D}(G)$ if there is a compact open subgroup $K \leq G$ such that $f(xy)=f(x)$ for all $y \in K$, and one can check that this is equivalent to ask the same condition for some of the subgroups $G_l$, $l\in \N_0$. We call a natural number $l$ the index of local constancy of $f$, and we write $ind(f)=l$, if $l$ is the minimum natural number such that $f(xy)=f(x)$ for all $y \in G_l$. 

The second class is the collection of Schwartz functions on $G$ that we define here, initially, via Fourier transform. We say that $f \in \mathcal{S}(G)$ if for all $m>0$ it holds $$\| \widehat{f}(\xi) \|_{HS} \lesssim \langle \xi \rangle^{-m}_{\mathscr{G}}, \esp \esp \esp \text{for all} \esp [\xi] \in \widehat{G}.$$

\subsection{Fundamental solution: compact case}
Let $G$ be a compact constant-order Vilenkin group together with the sequence of compact open subgroups $\mathscr{G}$, let us say $|G_k /G_{k+1}|= \varkappa$ for all $k \in \N_0$. The Vladimirov-Taibleson operator on $G$ takes the form $$D^\alpha u(x) = \frac{1-\varkappa^{\alpha}}{1 -\varkappa^{-\alpha -1}}  \int_{G}|y|_{\mathscr{G}}
^{-\alpha - 1}(u(xy^{-1})-u(x))dx.$$By now we will not focus on this operator and instead we want to consider the following closely related operator on $G$: $$\mathbb{D}^\alpha u (x) := \frac{1-\varkappa^{-1}}{1 -\varkappa^{-\alpha - 1}} u(x) +  \frac{1-\varkappa^{\alpha}}{1 -\varkappa^{-\alpha -1}}  \int_{G}|y|_{\mathscr{G}}
^{-\alpha - 1}(u(xy^{-1})-u(x))dx.$$We can express it as a right-convolution operator and we will call its convolution kernel,  the \emph{Riesz kernel} associated to $\mathscr{G}$. This kernel is determined by the complex-valued function $$\mathfrak{r}_s^\mathscr{G} (x) :=  \frac{|x|_{\mathscr{G}}^{s-1}}{\Gamma_{\mathscr{G}}(s)},\esp \esp \esp \esp  s \notin 1+\frac{2 \pi i}{\ln{\varkappa}} \Z, $$where $$ \Gamma_{\mathscr{G}} (x)= \frac{1 -\varkappa^{s-1}}{1-\varkappa^{-s}} , \esp s \neq 0.$$  As a distribution, the Riesz kernel has a meromorphic continuation to $\C$ given by
\begin{align*}
    \langle \mathfrak{r}_s^\mathscr{G} , f \rangle = \frac{1-\varkappa^{-1}}{1 -\varkappa^{s-1}} f(e) +  \frac{1-\varkappa^{-s}}{1 -\varkappa^{s-1}}  \int_{G}|x|_{\mathscr{G}}
^{s - 1}(f(x)-f(e))dx, \end{align*}with poles at $1+\frac{2 \pi i}{\ln{\kappa}} \Z$. In particular, for $\mathfrak{Re}(s)>0$ $$\langle \mathfrak{r}_s^\mathscr{G} , f \rangle =   \frac{1-\varkappa^{-s}}{1 -\varkappa^{s-1}}  \int_{G}|x|_{\mathscr{G}}
^{s - 1}f(x)dx.$$
\begin{rem}\label{remfouriertransformdistributions}
Given a distribution $ h \in \mathcal{D}'(G)$, its Fourier transform $\mathcal{F}_G h \equiv \widehat{h}\in \mathcal{D}'(\widehat{G}) $ is defined via the formula $$\langle \mathcal{F}_G h , \widehat{\varphi} \rangle := \langle h , \iota \circ \mathcal{F}_G^{-1} \widehat{\varphi} \rangle, \esp \esp \esp \widehat{\varphi} \in \mathcal{D} (\widehat{G}),$$where $(\iota \circ f)(x):= f(x^{-1})$. We can write explicitly $$\iota \circ \mathcal{F}_G^{-1} \widehat{\varphi}(x) = \sum_{[\xi]\in \widehat{G}} d_\xi Tr\big( \xi^* (x) \widehat{\varphi}(\xi)\big), \esp \esp \esp \langle h , \iota \circ \mathcal{F}_G^{-1} \widehat{\varphi} \rangle = \sum_{[\xi]\in \widehat{G}} d_\xi Tr\big( \langle h, \xi^* (x) \rangle \widehat{\varphi}(\xi)\big).$$In this way we see that the action of the Fourier transform of a distribution is determined by its values in the matrix coefficients $\xi_{ij}$. We can identify $\mathcal{F}_G h $ with the map $\widehat{h} \in \mathcal{M}(\widehat{G})$ defined in each $[\xi] \in \widehat{G}$ as $\widehat{h}(\xi) := \langle h , \xi^* \rangle.$  
\end{rem}

With a straightforward calculation we can check that $$\widehat{\mathfrak{r}}_s^{\mathscr{G}} (\xi) = \langle \xi \rangle^s_{\mathscr{G}} I_{d_\xi} , \esp [\xi] \in \widehat{G},$$which shows that $\mathfrak{r}_s^{\mathscr{G}} * \mathfrak{r}_t^{\mathscr{G}} = \mathfrak{r}_{s+t}^{\mathscr{G}}$ and in particular $$\langle \mathfrak{r}_s^{\mathscr{G}} * \mathfrak{r}_{-s}^{\mathscr{G}} , f \rangle = \langle \mathfrak{r}_0^{\mathscr{G}} , f \rangle = \lim_{s \to 0} \langle \mathfrak{r}_s^{\mathscr{G}},f \rangle = \langle  \delta_e, f  \rangle .$$Using this information we can compute a fundamental solution for $\mathbb{D}^\alpha.$
\begin{defi}\label{defifundamsolandparametrix}\normalfont
Let $T: \mathcal{D}(G) \subset D(T) \to \mathcal{D}'(G)$ be a densely defined linear operator. We say that $F \in \mathcal{D}'(G)$ is a \emph{fundamental solution} of $T$ if $T * F= \delta_e$. 
\end{defi}
\begin{teo}\label{teofundamentalsolVTop}
A fundamental solution for the VT-type operator $\mathbb{D}^\alpha$, $\mathfrak{Re}(\alpha)>0 , \mathfrak{Im}(\alpha) \notin (2 \pi i / \ln \varkappa) \Z$, is given by \[E_\alpha (x) = \begin{cases} \frac{1 - \varkappa^{-\alpha}}{1 - \varkappa^{\alpha -1 }} | x|_\mathscr{G}^{\alpha -1} & \esp \esp \text{ if } \esp \esp \mathfrak{Re}(\alpha) \neq 1, \\ \frac{1 - \varkappa}{\varkappa \ln \varkappa} \ln{(|x|_\mathscr{G})} & \esp \esp \text{if} \esp \esp \mathfrak{Re}(\alpha)=1 .\end{cases}\] 
\end{teo}
\begin{proof}
According to Definition \ref{defifundamsolandparametrix} we are looking for a distribution $E_\alpha \in \mathcal{D}'(G)$ such that $$\langle \xi \rangle^{\alpha} \mathcal{F}E_\alpha = 1_{\widehat{G}}\esp 
, \esp \esp \langle 1_{\widehat{G}} , \widehat{f} \rangle := \sum_{[\xi] \in \widehat{G}} d_\xi Tr( \widehat{f}(\xi)),$$where $1_{\widehat{G}} := \mathcal{F}(\delta_e)$. To find it we will use the Laurent expansion at $-\alpha$ of the distribution $\langle \xi \rangle^s$. Let us write $$\langle \xi \rangle^{s} = \sum_{m \in \Z} c_m (s+\alpha)^m.$$The real parts of the poles of the meromorphic continuation of $\langle \xi \rangle^s$ are negative rational numbers so $\langle \xi \rangle^{s+\alpha}= \mathcal{F}_G (\mathfrak{r}_{s+\alpha}^\mathscr{G})$ is holomorphic at $s=-\alpha$. This implies $c_0 = 0$ for $m < 0$ so $$\langle \xi \rangle^{s +\alpha} = c_0 \langle \xi \rangle^\alpha +  \sum_{m \in \N} c_m (s+\alpha)^m \langle \xi \rangle^\alpha (s+\alpha)^m.$$By the Lebesgue dominated convergence theorem $$ \lim_{s \to - \alpha} \langle \langle \xi \rangle^{s+\alpha} , \widehat{f} \rangle = \lim_{s \to - \alpha} \langle \mathfrak{r}_s^\mathscr{G} , \iota \circ f  \rangle = \langle \delta_e , f \rangle,$$therefore we can choose $E_\alpha$ as the distribution $c_0$ in the Laurent expansion of $\langle \xi \rangle^s$. Moreover, if $- \alpha$ is not a pole of $| x|_\mathscr{G}^s$ we have $$\mathcal{F} E_\alpha = \lim_{s \to - \alpha} \langle \xi \rangle^s .$$Now, in order to compute $E_\alpha$, we need to consider two cases.\begin{itemize}
    \item If $\mathfrak{Re}(\alpha) \neq 1$, we know that $$\langle \langle \xi \rangle^s , \widehat{f} \rangle = \frac{1 - \varkappa^s}{1 - \varkappa^{-(s+1)}} \int_G | x|_\mathscr{G}^{- (s +1)} f(x^{-1}) dx=\int_G | x|_\mathscr{G}^{- (s +1)} f(x) dx,$$ So, by the Lebesgue dominated convergence theorem, the following interchange of the integral with the limit is allowed \begin{align*}
        \lim_{s \to - \alpha}\langle \langle \xi \rangle^s , \widehat{f} \rangle &= \lim_{s \to -\alpha} \frac{1 - \varkappa^s}{1 - \varkappa^{-(s+1)}} \int_G | x|_\mathscr{G}^{- (s +1)} f(x) dx \\ &= \int_G \lim_{s \to -\alpha} \frac{1 - \varkappa^s}{1 - \varkappa^{-(s+1)}}  | x|_\mathscr{G}^{- (s +1)} f(x) dx \\ &= \frac{1 - \varkappa^{-\alpha}}{1 - \varkappa^{\alpha-1}} \int_G | x|_\mathscr{G}^{\alpha - 1} f(x) dx.
    \end{align*}
    \item If $\mathfrak{Re}(\alpha)=1$, we rewrite \begin{align*}
        \langle \langle \xi \rangle^s , \widehat{f} \rangle &= \frac{1 - \varkappa^s}{1 - \varkappa^{-(s+1)}} \int_G | x|_\mathscr{G}^{- (s +1)} f(x) dx \\ &= \int_G \frac{1 - \varkappa^s}{1 - \varkappa^{-(s+1)}}  \varkappa^{v_\mathscr{G}(x) ( s +1)} f(x) dx,
    \end{align*}where $v_\mathscr{G}(x) = n$ if $x \in G_n \setminus G_{n+1}.$ Using again the Laurent series we get \begin{align*}
        \frac{1 - \varkappa^s}{1 - \varkappa^{-(s+1)}}  \varkappa^{v_\mathscr{G}(x) ( s +1)} &= \frac{1 - \varkappa^{-1}}{\ln \varkappa} (s+1)^{-1}\\&+ \frac{(1 - \varkappa^{-1}) \ln \varkappa^{v_\mathscr{G}(x)} - \ln \varkappa /\varkappa + (\varkappa-1/2 \varkappa)\ln \varkappa}{\ln \varkappa}\\& + \mathcal{O}(s+1),
    \end{align*}and therefore $$ \lim_{s \to - \alpha} \langle \langle \xi \rangle^s , f \rangle = \int_G  \frac{(1 - \varkappa^{-1}) \ln |x|_\mathscr{G} - \ln \varkappa /\varkappa + (\varkappa-1/2 \varkappa)\ln \varkappa}{\ln \varkappa} f(x) dx.$$Finally, since the sum of a fundamental solution and a constant is also a fundamental solution, we get $$  \langle E_\alpha  , f \rangle = \int_G  \frac{(1 - \varkappa^{-1}) \ln |x|_\mathscr{G} }{\ln \varkappa} f(x) dx.$$ 
\end{itemize}
The proof is concluded.
\end{proof}
\begin{rem}\label{remtwometrics}
Consider again a locally compact Vilenkin group $G$, endowed with a strictly decreasing sequence of compact open subgroups $\mathscr{G}=\{G_n\}_{n \in \Z}$. As we mentioned before, each one of the compact open subgroups $G_k$ is a compact Vilenkin group with the sequence $\mathscr{G}_k=\{G_{k+n}\}_{n \in \N_0}$, and for each one of this groups we can define the operator $$\mathbb{D}^\alpha_k u (x) := \frac{1-\varkappa^{-1}}{1 -\varkappa^{-\alpha - 1}} u(x) +  \frac{1-\varkappa^{\alpha}}{1 -\varkappa^{-\alpha -1}}  \int_{G_k}|y|_{\mathscr{G}_k}
^{-\alpha - 1}(u(xy^{-1})-u(x))d_k x,$$where $d_k x$ denotes the normalised Haar measure on $G_k$ which, in terms of the normalised measure $dx$ on $G$, is given by $d_k x = \varkappa^k dx$. Also we can put the ultrametric on $G_k$ in terms of the ultrametric on $G$ since for any $x \in G_k$ it holds $|x|_{\mathscr{G}} = \varkappa^{-k} |x|_{\mathscr{G}_k} $. 
\end{rem}

\subsection{Locally compact Vilenkin groups}
Now we recall some definitions and terminology associated to locally profinite groups that we will use for the rest of this work. 
\begin{defi}\label{defismoothrep}\normalfont 
    Let $G$ be a locally compact totally disconnected group.
    \begin{itemize}
        \item As before, given a measurable set $A\subset G$, we will denote by $\1_A$ the characteristic function of $A$. Also, we define the function $\epsilon_A$ as: $$\epsilon_A (x) := \frac{1}{|A|} \1_A (x).$$
        \item We denote by $\mathcal{D} (G)$ the collection of locally constant functions on $G$ with compact support. If $G$ is a locally compact Vilenkin group endowed with a sequence of subgroups $\{ G_n\}_{n \in \Z}$, we use the notation $$\mathcal{D} (G_n) := \{ f \in \mathcal{D} (G) \esp : \esp Supp(f) \subseteq G_n\},$$ $$\mathcal{D}_l (G) := \{f \in \mathcal{D} (G) \esp : \esp f(xy)=f(x), \esp \esp \forall y \in G_l\},$$ $$\mathcal{D}_l (G_n) := \{f \in \mathcal{D} (G_n) \esp : \esp f(xy)=f(x), \esp \esp \forall y \in G_l\}, \esp \esp l \geq n.$$Also, we denote by $\tilde{\mathcal{D}}(G) $ the collection of functions $f \in \mathcal{D}(G)$ such that $\int_G f (x)dx =0$. We will use the notation $$\tilde{\mathcal{D}} (G_n) := \{ f \in \tilde{\mathcal{D}} (G) \esp : \esp Supp(f) \subseteq G_n\},$$ $$\tilde{\mathcal{D}}_l (G) := \{f \in \tilde{\mathcal{D}} (G) \esp : \esp f(xy)=f(x), \esp \esp \forall y \in G_l\},$$ $$\tilde{\mathcal{D}}_l (G_n) := \{f \in \tilde{\mathcal{D}} (G_n) \esp : \esp f(xy)=f(x), \esp \esp \forall y \in G_l\}, \esp \esp l \geq n.$$The symbol $\mathcal{D}'(G)$ will stand for the distributions on $\mathcal{D} (G)$ and we will use $\tilde{\mathcal{D}}'(G)$ to denote the space of distributions on $\tilde{\mathcal{D}}(G)$.  
        \item We will use the symbol $C^\infty (G)$ to denote the collection of locally constant distributions on $G$ with a fixed index of locally constancy, that is, $F \in C^\infty (G)$ if $F \in \mathcal{D}'(G)$ and there exists an $l \in \Z$ such that $\tau_y F=F $ for all $y \in G_l$, where the translation $\tau_y$ is defined as $$\langle \tau_y F, f \rangle := \langle F , f (\cdot y^{-1}) \rangle.$$We will also use the notation $$C^\infty_l (G) := \{ F \in C^\infty(G) \esp : \esp \tau_y F =F \esp \text{for } \esp y \in G_l \}.$$
    \end{itemize}
\end{defi}

\begin{rem}\label{RemtopologyD(G)}
In order to make $\mathcal{D}(G)$ a topological vector space, we will say that a sequence $\{\varphi_j \}_{j \in \N} \subset \mathcal{D}(G)$ converges to $0$ if and only if: 
\begin{enumerate}
    \item There exists $l,n \in \Z$ such that $\{\varphi_j\}_{j \in \N} \subseteq \mathcal{D}_l (G_n)$.
    \item $\varphi_j \to 0$ uniformly when $j \to \infty$.
\end{enumerate}With this topology we find that every linear functional on $F : \mathcal{D} (G) \to \C$ is continuous or, in other words, the dual of $\mathcal{D}(G)$ as topological vector space coincide with its algebraic dual. 
\end{rem}

Now we recall briefly some elements of the representation theory of t.d. groups that we will use in what follows. 
\begin{defi}
    Let $G$ be a locally compact Vilenkin group with a sequence of compact open subgroups $\{G_n \}_{n \in \Z}$.
    \begin{itemize}
        \item A representation $(\pi , V)$ is called a \emph{smooth representation} if the stabilizer in $G$ of any vector $\varphi \in V$ is open.
        \item A smooth representation $(\pi, V)$ is called \emph{admissible} if for every open subgroup $H \leq G$ the subspace $$V^H:=\{ \varphi \in V \esp : \esp \pi (h) \varphi = \varphi \esp \esp \text{for all} \esp h \in H \} \subset V$$is finite dimensional.
        \item We will denote by the set of all equivalence classes of (algebraically) irreducible admissible representations of $G$ by $\tilde{G}$. The set of all equivalence classes of topologically irreducible unitary representations (on  Hilbert  spaces) is denoted by $\widehat{G}$. It is well known that the unitary dual $\widehat{G}$ is in a natural bijection with the subset of  all unitarizable classes in $\tilde{G}$. In this way we shall identify $\widehat{G}$ with the subset of all unitarizable classes in $\tilde{G}$. Also, we will use $Rep(G)$ to denote the collection of smooth, admissible, unitary representations of $G$.
        \item Let $(\pi ,V_\pi)$ be an smooth admissible representation of $G$. We define the Fourier transform of $f \in L^1 (G) \cap L^2 (G)$ with respect to $(\pi, V_\pi)$ as the linear operator $\widehat{f}(\pi)$ acting on $V_\pi$ via the formula $$\mathcal{F}_G[f](\pi) = \widehat{f}(\pi)v = \int_G f(x) \pi^* (x) v \esp dx, \esp \esp \esp v \in V_\pi. $$The measurable field $\{\widehat{f}(\pi)\}_{\pi \in \widehat{G}}$ will be called the Fourier transform of $f$. 
        \item We will denote by $\mathcal{D}(\widehat{G})$ the collection of measurable fields $\{\sigma(\pi) \}_{\pi \in \widehat{G}}$ such that there is an $l \in \Z$ so that for almost all $[\pi] \in \widehat{G}$ it holds $Im \esp \sigma(\pi ) \subseteq \mathcal{H}_\pi^{G_l}$.
    \end{itemize}
    \end{defi}
    
\begin{teo}\label{plachereltheolccase}
Let $G$ be a locally compact Vilenkin group together with the sequence of compact open subgroups $\{G_n\}_{n \in \Z}$. Let $dx$ the Haar measure on $G$ that makes $|G_0|=1$. Then, the unitary dual $\widehat{G}$ of $G$ is equipped with a measure $\nu$, called the Plancherel measure, such that: the operator $\widehat{f}(\pi)$ is of trace class for any $f \in \mathcal{D}(G)$ and any smooth admissible irreducible representation $\pi \in \widehat{G}$, and $Tr(\widehat{f} (\pi))$  depends only on the class of $\pi$. The function $\pi \mapsto Tr(\widehat{f} (\pi))$ is integrable against $\nu$ and the following formula holds: 
$$f(e)=\int_{\widehat{G}} Tr(\widehat{f}(\pi)) d\nu (\pi).$$In particular, for every $x \in G$ the map $ \pi \mapsto Tr(\pi(x)\widehat{f} (\pi) ) $ is integrable and $$f(x) = \int_{\widehat{G}} Tr(\pi(x)\widehat{f} (\pi) ) d \nu (\pi)$$ 
\end{teo}
\begin{rem}\label{remfourierdistributions}
We can extend our definition of Fourier transform to the elements of $\mathcal{D}'(G)$ in the following way: for $F \in \mathcal{D}'(G)$ define define its Fourier transform $\mathcal{F}_G F \in \mathcal{D}(\widehat{G})$ as $$\langle \mathcal{F}_G [F] , \widehat{h} \rangle := \langle F , \iota \circ \mathcal{F}_G^{-1} f \rangle, \esp \esp h \in \mathcal{D}(\widehat{G}),$$where $\iota \circ f (x) := f (x^{-1})$. 
\end{rem}

\subsection{Fundamental solution: non-compact case}
Now that we know the fundamental solution in the compact case we can move to the the locally compact case. Let us take $f \in \mathcal{D} (G_n)$ so $f(x)=0$ for $x \in G \setminus G_n.$ Recall that $\1_A$ denotes the characteristic function in $A$. 
We want to to prove that $E_\alpha (x)$ in Theorem \ref{teofundamentalsolVTopLC} is a fundamental solution for $D^\alpha$, that is $D^\alpha * E_\alpha = \delta_e$ on $ \mathcal{D} '(G)$. To do that we will show that almost everywhere on $G$ and for any $\varepsilon>0$ we have $|f * D^\alpha * E_\alpha(x) - f(x)|  < \varepsilon$ so that $f * D^\alpha * E_\alpha(x) = f(x)$. The argument here is that we can write the action of the Vladimirov-Taibleson operator on $f$ in terms of the action of the operators $\mathbb{D}^\alpha$ in the following way \begin{align*}
    \int_G \frac{f(xy^{-1}) - f(x)}{|y|_{\mathscr{G}}^{ \alpha + 1}} dy &= \1_{G_n} (x) \int_{G_n} \frac{f(xy^{-1}) - f(x)}{|y|_{\mathscr{G}}^{ \alpha + 1}} dy   + \1_{ G_n} (x) \int_{G\setminus G_n} \frac{f(xy^{-1}) - f(x)}{|y|_{\mathscr{G}}^{ \alpha + 1}} dy \\ & \esp + \1_{G \setminus G_n} (x) \int_{G_n} \frac{f(xy^{-1}) - f(x)}{|y|_{\mathscr{G}}^{ \alpha + 1}} dy + \1_{G \setminus G_n} (x) \int_{G \setminus G_n} \frac{f(xy^{-1}) - f(x)}{|y|_{\mathscr{G}}^{ \alpha + 1}} dy.
\end{align*}
Now we proceed by parts: 
\begin{itemize}
    \item The first integral can be rewritten as: $$\1_{G_n} (x) \int_{G_n} \frac{f(xy^{-1}) - f(x)}{|y|_{\mathscr{G}}^{ \alpha + Q}} dy = \varkappa^{n \alpha} \1_{G_n} (x) \int_{G_n} \frac{f(xy^{-1}) - f(x)}{|y|_{\mathscr{G}_n}^{ \alpha + Q}} d_ny.$$
    \item If $x \in G_n$ and $y \in G \setminus G_n$, then $x y^{-1} \in G \setminus G_n$ so $$\1_{ G_n} (x) \int_{G\setminus G_n} \frac{f(xy^{-1}) - f(x)}{|y|_{\mathscr{G}}^{ \alpha + 1}} dy = - f(x) \int_{G \setminus G_n} |y|_{\mathscr{G}}^{-(\alpha + 1)} dy = \frac{1 - \varkappa^{-1}}{1 - \varkappa^\alpha} \varkappa^{n \alpha} f(x). $$
    \item If $x \in G \setminus G_n$ and $y \in G_n$ then $f(xy^{-1}) = f(x) =0$, so $$\1_{G \setminus G_n} (x) \int_{ G_n} \frac{f(xy^{-1}) - f(x)}{|y|_{\mathscr{G}}^{ \alpha + 1}} dy=0.$$
    \item If $x \in G \setminus G_n$ and $y \in G \setminus G_n$,  the integral is non-zero if and only if $xy^{-1} \in G_n$, and when this happens $|x|_G = |y|_G$. Then we get \begin{align*}
        \1_{ G \setminus G_n} (x) \int_{G\setminus G_n} \frac{f(xy^{-1}) - f(x)}{|y|_{\mathscr{G}}^{ \alpha + 1}} dy &= \1_{G \setminus  G_n} (x) \int_{G \setminus G_n } \frac{f(xy^{-1}) }{|y|_{\mathscr{G}}^{ \alpha + 1}} dy \\ &= \1_{G \setminus  G_n} (x) |x|_{\mathscr{G}}^{-(\alpha +1)} \int_{x G_n} f(xy^{-1}) dy \\ & = \1_{G \setminus  G_n} (x) |x|_{\mathscr{G}}^{-(\alpha +1)} \int_{ G_n} f(y) dy.
    \end{align*}
\end{itemize}
Summing up: 
\begin{align*}
    D^\alpha f(x) = \varkappa^{n \alpha} \mathbb{D}_n^{\alpha} f(x) +  \frac{1 - \varkappa^\alpha}{1 - \varkappa^{- (\alpha + 1)}} \1_{G \setminus  G_n} (x) |x|_{\mathscr{G}}^{-(\alpha +1)} \int_{G_n} f(y) dy,\end{align*} and in particular for $f \in \tilde{\mathcal{D}}(G_n)$ we have $D^\alpha f = \varkappa^{ \alpha} \mathbb{D}^\alpha_n f$. Moreover, the previous equalities hold true if we replace $n$ for any $l$ such that $l \leq n$. Using this calculation we are now in position to prove Theorem \ref{teofundamentalsolVTopLC}. 
\begin{proof}[Proof of Theorem \ref{teofundamentalsolVTopLC}:]
Let us start by remarking that for $f \in \tilde{\mathcal{D}}(G_n)$ it holds $D^\alpha f = \varkappa^{l \alpha} \mathbb{D}_l^\alpha f$ for any $l \leq n$. If we fix $x\in G$ and choose $l \leq n$ so that $x \in G_l$ we obtain  \begin{align*}
    D^\alpha f * E_\alpha ( x) &=\frac{1 - \varkappa^{-\alpha}}{1 - \varkappa^{\alpha-1}} \varkappa^{l \alpha} \int_G \mathbb{D}_l^\alpha f (y) |y^{-1} x|_{\mathscr{G}}^{\alpha -1} dx \\ & =\frac{1 - \varkappa^{-\alpha}}{1 - \varkappa^{\alpha-1}}\varkappa^{l \alpha} \int_{G_l} \mathbb{D}_l^\alpha f (y) |y^{-1} x|_{\mathscr{G}}^{\alpha -1} dx \\ & =\frac{1 - \varkappa^{-\alpha}}{1 - \varkappa^{\alpha-1}} \varkappa^{l \alpha} \int_{G_l} \mathbb{D}_l^\alpha f (y)(\varkappa^{-l(\alpha -1)}) |y^{-1} x|_{\mathscr{G}_l}^{\alpha -1} dx \\ &=\frac{1 - \varkappa^{-\alpha}}{1 - \varkappa^{\alpha-1}}   \int_{G_l} \mathbb{D}_l^\alpha f (y) |y^{-1} x|_{\mathscr{G}_l}^{\alpha -1} d_l  x= f(x).
\end{align*}
For more general $f \in \mathcal{D} (G_n)$ and $l \leq n$ something similar holds and we only need to observe that $$\frac{1 - \varkappa^{-\alpha}}{1 - \varkappa^{\alpha-1}}\frac{1 - \varkappa^\alpha}{1 - \varkappa^{- (\alpha + 1)}} \int_{G} \1_{G \setminus  G_l} (z) |z|_{\mathscr{G}}^{-(\alpha +1)} |z^{-1} x|_{\mathscr{G}}^{\alpha -1} dz \int_{G_n} f(y) dy \to 0,$$as $l \to - \infty$. This concludes the proof.   
\end{proof}
\subsection{The isotropic Laplace operator}
To conclude this section we give a quick review of some of the results exposed in \cite{DirichletForms, IsoMarkovSemi} and the way they can be applied to graded $\K$-Lie groups. Specially, we are interested in the heat kernel estimates obtained in \cite{IsoMarkovSemi} for the isotropic Laplace operator. Since the arguments in \cite{IsoMarkovSemi} work on general ultrametric measure spaces we can apply them in particular to graded $\K$-Lie groups and obtain in this way an estimate for the heat kernel of the Vladimirov-Taibleson operator and some other interesting homogeneous operators. We will focus exclusively on graded $\K$-Lie groups but our analysis in this section extends naturally to constant-order Vilenkin groups.

Let $G$ be a graded $\K$-Lie group and consider the ultrametric measure space $(G, d_\alpha , dx)$, where $dx$ is the normalised Haar measure on $G$ and $$d_\alpha (x,y):= \frac{|y^{-1}x|_G^{\alpha}}{q^\alpha}, \esp \esp \alpha >0,$$so that the balls in this metric have the property $$ |B_{\alpha}(x,r)| \asymp r^{\frac{Q}{\alpha}}.$$Our goal is to get some estimate on the heat kernel associated to $\mathscr{D}^\alpha$, $\alpha>0$, but, instead of starting with $\mathscr{D}^\alpha$ and its heat semigroup, the authors in \cite{IsoMarkovSemi} begin by constructing a Markov semigroup associated to a triple $(\rho, d_\alpha , dx)$, where $\rho$ is a function $\rho : [0, \infty] \to [0,1]$ strictly monotone increasing  and left-continuous such that $\rho (0+)=0$ and $\rho(\infty )=1$. The heat kernel associated to this triple has in fact several very nice explicit forms that can be used to obtain the desired estimates, as it is shown in \cite[Theorem 2.14]{IsoMarkovSemi}. In particular, if we pick the function $\rho (r) := e^{- r^{-1}}$, the infinitesimal generator of the Markov semigroup associated to the triple $(\rho, d_\alpha , dx)$ coincides with $\mathscr{D}^\alpha$, so that their heat kernels coincide and we obtain the estimates for the heat kernel of $\mathscr{D}^\alpha$. For the sake of completeness, let us give more details about this construction.

Define the family of orthoprojectors $\{ P_r \}_{r \in [0,\infty )}$ by  $$P_r f (x) = \frac{1}{|B_\alpha (x,r)|} \int_{B_\alpha (x,r)} f (y) dy, \esp \esp r>0,  $$and $P_0 = 0$, where $B_\alpha (x,r)$ denotes the ball of radius $r >0$ with center in $x \in G$. Then for any function $\rho : [0, \infty] \to [0,1]$ strictly monotone increasing  and left-continuous such that $\rho (0+)=0$ and $\rho(\infty )=1$, the formula $$P f (x) := \int_0^\infty P_r f (x) d \rho (r), $$defines a Markov operator on Borel bounded functions, and a bounded operator on $L^2 (G)$. The non-negative powers of this operator are given by $$P^t f (x) := \int_0^\infty P_r f (x) d \rho^t (r), $$and they form a strongly continuous symmetric Markov semigroup on $L^2 (G)$. If we pick the function $\rho (r) := e^{- r^{-1}}$ then for any $t>0$ these operators admit an integral kernel called the \emph{
heat kernel of the semigroup}. For this kernel the following expression is known: $$ h_\alpha (t,x,y)= \int_{d_\alpha (x,y)}^\infty \frac{ d \rho^t (r) }{|B_\alpha (x,r)|},$$and the following estimate holds true \cite[Theorem 2.14]{IsoMarkovSemi}:$$h_\alpha (t,x,y) \asymp \frac{t}{(t^{1/\alpha} + |y^{-1}x|_G)^{ \alpha + Q }}.$$We can construct a Dirichlet form $(\mathcal{E}, Dom_\mathcal{E} )$ from the heat kernel $h_\alpha$ via the formula $$\mathcal{E}(f,f):= \lim_{t \to 0} \frac{1}{2t} \int_G \int_G h_\alpha (t,x,y)(f(x) - f(y))^2 dx dy,$$whose generator coincides with the infinitesimal generator of the semigroup $P_\rho^t$, which is a self-adjoint unbounded operator $\mathcal{L}_\alpha$ on $L^2 (G)$ given by $$\mathcal{L}_\alpha f = \lim_{t \to 0 } \frac{P^t f - f}{t}.$$The authors in \cite{IsoMarkovSemi} refer to this operator as the \emph{isotropic Laplace operator} associated with $(\rho, d_\alpha, dx )$. Using the polarization identity we can write this non-local Dirichlet forms in terms of a kernel $J_\alpha(x,y)$ as $$\mathcal{E}(f,g) = \frac{1}{2} \int_G \int_G(f(x) - f(y))(g(x) - g(y)) J_\alpha(x,y) dx dy ,$$with the  kernel $J_\alpha(x,y)$ given by \begin{align*}
    J_\alpha(x,y) &= \int_{d_\alpha (x,y)}^\infty \frac{s^{-2} ds}{ |B_\alpha ( x,s)|} \\ &= \sum_{n=0}^\infty \frac{1}{q^{Q(n+k)}} \int_{q^{\alpha(k+n)}/q^\alpha}^{q^{\alpha(k+n+1)}/q^\alpha} d(-1/r) \\ &=\frac{q^\alpha}{q^{(Q+\alpha)k}}\sum_{n=0}^\infty \frac{1}{q^{n(Q + \alpha)}}\big(1-\frac{1}{q^\alpha} \big) = - \frac{1-q^{\alpha}}{1-q^{-(\alpha + Q)}} | y^{-1}x|_G^{-(\alpha + Q)},
\end{align*}where in the above calculations we took $d_\alpha (x,y) = q^{\alpha k}/q^\alpha$ fixed. According to \cite[Theorem 3.2]{IsoMarkovSemi} the isotropic Laplace operator $\mathcal{L}_\alpha$ acts on functions $f \in \mathcal{D}(G)$ as $$\mathcal{L}_\alpha f (x) = \int_G f(x) - f(y) J_\alpha (x,y) dy,$$so we actually have the equality $\mathcal{L}_\alpha = \mathscr{D}^\alpha$, the heat kernel $h_\alpha (t,x,y) $ is the heat kernel of $\mathscr{D}^\alpha$ and it actually has the form $h_\alpha (t,x,y)= h_\alpha (t,y^{-1}x) $. We can try to obtain the fundamental solution for $\mathscr{D}^\alpha$ by integrating $h_\alpha$ in the variable $t\in (0 , \infty)$, but in fact this works if and only if the heat semigroup is transient which is equivalent to $0<\alpha <Q $. In such case it holds $$E_\alpha (x) = \int_0^\infty h_\alpha (t,x) dt = \frac{1-q^{-\alpha}}{1-q^{\alpha - Q}} |x|_G^{\alpha - Q},$$which gives an alternative proof of Theorem \ref{fundamentalsolutiongradedliegroups} for the special case $0<\alpha<Q$.

\section{The Heisenberg group}

\subsection{Homogeneous operators on $\Q_p^d$}
As we remarked in the introduction, the Vladimirov-Taibleson operator defines an homogeneous operator on graded $\K$-Lie groups. Of course there are many other interesting examples of homogeneous operators, and it is our purpose to introduce some of them along this section. A nice way to explain our ideas is by revisiting the commutative case, so we begin by recalling some properties of invariant $\nu$-homogeneous operators on $\Q_p^d$, $\nu >0$. Via Fourier analysis invariant $\nu$-homogeneous operators, that is operators satisfying $T_\sigma (f(\gamma x)) = |\gamma|_p^\nu T_\sigma f ( \gamma x)$ for $\gamma \in \Q_p^*$, can be written in the form $$T_\sigma  f (x) = \int_{\Q_p^d} \sigma (\xi) \widehat{f}(\xi) e^{2 \pi i \{x \cdot \xi  \}_p}d \xi,$$where $\sigma : \Q_p^d \to \C$ is a $\nu$-homogeneous function. If $T_\sigma$ is positive, elliptic and homogeneous of positive degree, which is equivalent to $\sigma (\xi ) \geq 0$ and $\sigma (\xi)=0$ iff $\xi = 0$, we can extend $T_\sigma$ to a self adjoint operator and we can use its associated functional calculus to define the heat semigroup $$e^{- t T_\sigma} f (x) = \int_{\Q_p^d} e^{- t \sigma (\xi)} \widehat{f}(\xi) e^{2 \pi i \{x \cdot \xi  \}_p}d \xi,$$and obtain from its kernel the fundamental solution and the Riesz potential associated to $T_\sigma$. The simplest example of this construction occurs with the Vladimirov-Taibleson operator $\mathscr{D}^\alpha$, an $\alpha$-homogeneous operator on $\Q_p^d$ which via Fourier transform can be expressed as $$\mathscr{D}^\alpha  f (x) = \int_{\Q_p^d} \| \xi \|_p^\alpha \widehat{f}(\xi) e^{2 \pi i \{x \cdot \xi  \}_p}d \xi.$$The heat kernel for this operator is given by $$h_\alpha (t,x) = \int_{\Q_p^d} e^{- t \| \xi \|_p^\alpha} \widehat{f}(\xi) e^{2 \pi i \{x \cdot \xi  \}_p}d \xi,$$ and it is studied in \cite{RodriguezJ}, where the reader can find a detailed analysis of the Cauchy problem associated to the Taibleson operator. The same analysis and similar estimates as in \cite{RodriguezJ} hold true for more general positive elliptic homogeneous operators, and we can use their heat kernels to obtain, at least formally, their fundamental solutions in an standard way. That is, if $h_\sigma (t,x)$ is the heat kernel of $T_\sigma$, a fundamental solution for $T_\sigma$ is given by $$\textbf{h}_\sigma (x) = \int_0^\infty h_\sigma (t,x)dt.$$Moreover, under the hypothesis that the heat kernel satisfies some reasonable estimate, our guess is $h_\sigma (t,x) \asymp  t^{d/\nu}$ as for the Vladimirov-Taibleson operator, we can also define fractional powers of $T_\sigma$ and obtain their fundamental solution by means of the Riesz potential $$\mathcal{I}_a (x): = \frac{1}{\Gamma(a/\nu) } \int_0^\infty t^{a/\nu - 1 }h_\sigma (t,x) dt,  \esp \esp 0< \mathfrak{Re}(a) < d.$$

\begin{rem}
A remarkable property of these kernels is that they are homogeneous, as we can check by using the homogeneity of $T_\sigma$ and the properties of the heat kernel. First, for $\textbf{h}_\sigma$ we obtain with a couple of changes of variable \begin{align*}
    \textbf{h}_\sigma (\gamma x) & = \int_0^\infty \int_{\Q_p^d} e^{- t \sigma (\xi)}  e^{2 \pi i \{ \gamma x \cdot \xi  \}_p}d \xi dt \\  & = |\gamma|_p^{-d} \int_0^\infty \int_{\Q_p^d} e^{- t |\gamma|_p^{-\nu}\sigma (\xi)}  e^{2 \pi i \{  x \cdot \xi  \}_p}d \xi dt \\&= |\gamma|_p^{\nu-d} \int_0^\infty \int_{\Q_p^d} e^{- t \sigma (\xi)}  e^{2 \pi i \{  x \cdot \xi  \}_p}d \xi dt= |\gamma|_p^{\nu-d} \textbf{h}_\sigma ( x), 
\end{align*} 
Similarly we have for $\mathcal{I}_a$: \begin{align*}
    \mathcal{I}_a (\gamma x)&= \frac{1}{\Gamma(a/\nu) } \int_0^\infty \int_{\Q_p^d}  t^{a/\nu - 1 }e^{- t \sigma (\xi)}  e^{2 \pi i \{ \gamma x \cdot \xi  \}_p}d \xi dt \\ & =|\gamma|_p^{-d} \frac{1}{\Gamma(a/\nu) } \int_0^\infty \int_{\Q_p^d}  t^{a/\nu - 1 }e^{- t |\gamma|_p^{- \nu} \sigma (\xi)}  e^{2 \pi i \{  x \cdot \xi  \}_p}d \xi dt \\ & = |\gamma|_p^{a-d} \frac{1}{\Gamma(a/\nu) } \int_0^\infty \int_{\Q_p^d}  t^{a/\nu - 1 }e^{- t  \sigma (\xi)}  e^{2 \pi i \{  x \cdot \xi  \}_p}d \xi dt  = |\gamma|_p^{a-d}\mathcal{I}_a (x).
\end{align*}
\end{rem}
The same ideas exposed above for the case $G=\Q_p^d$ are applicable when we move to the noncommutative case, with some important modifications. The main one is that for noncommutative groups the representation theory, and therefore the Fourier analysis, is more complicated and for that reason it is necessary to appeal to more advanced arguments. Still we can study pseudo-differential operators on noncommutative groups and we will focus now our attention in a very particular class of homogeneous operators that we call here \emph{directional VT operators.} These are operators defined in terms of the elements of the Lie algebra of a graded $\K$-Lie group reassembling the definition of directional derivatives on Lie groups.  
\begin{defi}\normalfont
Let $G$ be a graded $\K$-adic Lie group with Lie algebra $\mathfrak{g}$. We define the \emph{directional VT operator} in the direction of $X \in \mathfrak{g}$ by the formula $$\partial_X^\alpha f (x) :=  \frac{1 - q^\alpha}{1 - q^{- (\alpha + 1)}} \int_{\K} \frac{f(x \cdot \mathbb{exp}(tX)^{-1}) - f(x)}{|t |_\K^{\alpha + 1}} dt .$$
\end{defi}

\begin{rem}
 A property we want to remark here is the homogeneity of the directional VT operators: for $\gamma \in \K$ and $X \in V_{\nu_k} \subset \mathfrak{g}$ we have \begin{align*}
    \partial_X^\alpha (f \circ D_\gamma) (u) &= \frac{1 - q^\alpha}{1 - q^{- (\alpha + 1)}} \int_{\K} \frac{f(\gamma(u \cdot  \mathbb{exp}( tX)^{-1})) - f(\gamma u)}{|t |_\K^{\alpha + 1}} dt \\ &= |\gamma|_\K^{\nu_k(\alpha + 1)} \frac{1 - q^\alpha}{1 - q^{- (\alpha + 1)}} \int_{\K} \frac{f(\gamma u \cdot  \mathbb{exp}( \gamma^{\nu_k} tX)^{-1})) - f(\gamma u)}{| \gamma^{\nu_k} t |_\K^{\alpha + 1}} dt\\ &= |\gamma|_\K^{\alpha \nu_k } \frac{1 - q^\alpha}{1 - q^{- (\alpha + 1)}}\int_{\K} \frac{f(\gamma u  \cdot  \mathbb{exp}( tX)^{-1}) ) - f(\gamma u)}{|  t |_\K^{\alpha + 1}} dt \\ &= |\gamma|_\K^{\alpha \nu_k} (\partial^\alpha_X f) \circ D_\gamma (u). 
\end{align*}
\end{rem}

It is important to remark that the assignation of a  VT operator to the elements of the Lie algebra is not linear, and therefore it doesn't preserve the Lie algebra structure. This is one of the biggest differences with the theory of real Lie groups: in the totally disconnected case we do not have an infinitesimal representation, and neither a correspondence between the Lie algebra and some notion of derivatives for functions $f: G \to \C$. Despite this fact the directional VT operators and the polynomials in the directional VT operators are interesting on their own, and there is a particular example that has been studied before in the literature called \emph{the Vladimirov Laplacian.} This operator is defined initially on $\mathcal{D}(\Q_p^d)$ as $$ \MO^\alpha := \sum_{j=1}^d \partial^\alpha_{e_j}=\sum_{j=1}^d \partial^\alpha_{x_j} ,$$where $e_1,...,e_d$ is the usual basis of $\Q_p^d$ and, as the authors in \cite[Section 5.3.1]{IsoMarkovSemi} pointed out, it has the following properties: 
\begin{itemize}
    \item $(\MO^\alpha, \mathcal{D}(\Q_p^d))$ is a non-negative symmetric, essentially self-adjoint operator.
    \item The semigroup $e^{-t \MO^\alpha}$ is symmetric and Markovian, and it admits a heat kernel $h_{\MO^\alpha} (t,x)$ with the property $$h_{\MO^\alpha} (t,x) \asymp t^{-d/\alpha}.$$
    \item The semigroup is transient if and only if $ 0< \alpha < d$. 
    \item For all $f \in \mathcal{D}(\Q_p^d)$ $$\MO^\alpha f (x) = \int_{\Q_p^d} (f(x) - f(y)) J_\alpha (x , dy),$$where $$J_\alpha (x,dy) = \sum_{j=1}^d J_\alpha (x_j - y_j) d j_j = - \frac{1 - p^\alpha }{1 - p^{-\alpha - 1}} \sum_{j=1}^d|x_j - y_j|_p^{-\alpha - 1} dy_j.   $$
\end{itemize}
As we mentioned in the introduction, for more general graded $\K$-Lie groups the question we want to investigate in this section is whether something similar occurs for an operator like $$\MO^\alpha :=\sum_{k=1}^r \sum_{j=1}^{b_k} \partial_{X_{k,j}}^{\alpha/\nu_k} ,$$where $\{ X_{k,j} \}$ is a Malcev basis associated to the gradation $$\mathfrak{g} = \bigoplus_{k=1}^{r} V_{\nu_k}, \esp \esp b_k := dim(V_{\nu_k}).$$We will show that the answer to this question is affirmative when $G= \mathbb{H}_d$ or $G=\mathbb{E}_4$ but first we need a quick review of some basic facts about harmonic analysis on $\mathbb{H}_d$ and $\mathbb{E}_4$.

\subsection{Harmonic analysis on the Heisenberg group}

One of the simplest noncommutative groups one can think about is the Heisenberg group introduced in Example \ref{exapadicheisenberg}. This $p$-adic Lie group has been extensively studied and we know explicitly a lot of its properties, for instance we can use the Kirilov orbit method to compute explicitly the unitary dual $\widehat{\mathbb{H}}_d (\Q_p)$ of $\mathbb{H}_d (\Q_p)$, or simply $\widehat{\mathbb{H}}_d$ and $\mathbb{H}_d$ for short. Unitary irreducible representations of the $p$-adic Heisenberg group follow the same pattern as in the real case. Given $[\pi] \in \widehat{\mathbb{H}}_d$ we have two possibilities: $\pi$ restricted to the center $\mathcal{Z}(\mathbb{H}_d)$ of $G=\mathbb{H}_d (\Q_p)$ is trivial, or it is not. In the first case, since $\mathbb{H}_d / \mathcal{Z}(\mathbb{H}_d)\cong \Q_p^{2d}$ is abelian, so the representation must have the form $$\pi_{\xi , \eta} (\textbf{x} , \textbf{y} , z) = e^{2 \pi i \{ \xi \cdot \textbf{x} + \eta \cdot \textbf{y} \}_p}, \esp \esp (\xi, \eta) \in \Q_p^{2d}.$$
In the second case it must hold $\pi (\textbf{x}, \textbf{y},z) |_{\mathcal{Z}(\mathbb{H}_d)} = e^{2 \pi i \{ \lambda z\}_p} $ for some $\lambda \in \Q_p^*$, and by the Kirillov orbith method we know that the $p$-adic Schrodinger representations $$\pi_\lambda (\textbf{x} , \textbf{y}, z) = e^{2 \pi i \{ \lambda (z +  \textbf{x} \cdot \textbf{y} + \textbf{y} \cdot u) \}_p}  \varphi (u + x), \esp \esp \varphi \in L^2 (\Q_p^d) ,$$ are unitary irreducible representations of $\mathbb{H}_d$ on the Hilbert space $\mathcal{H}_\lambda := L^2(\Q_p^d)$. Moreover, we can check that any unitary irreducible representation of $\mathbb{H}_d$ non trivial on the center is equivalent to some of the representations $\pi_\lambda$. In this way we can identify the unitary dual $\widehat{\mathbb{H}}_d$ with the set $\Q_p^* \cup \Q_p^{2d}$. There is a Borel measure on $\Q_p^* \cup \Q_p^{2d}$ that can be transferred to $\widehat{\mathbb{H}}_d$ and that assigns measure zero to the elements of $\Q_p^{2d}$. For that reason the relevant representations for us are the ones parameterized by $\Q_p^*$ and we will mostly think on the unitary dual of $\mathbb{H}_d $ as the collection of Schrodinger representations $\pi_\lambda.$ 

\begin{rem}\label{remdilationsheisenberg}
Schrodinger representations have several possible explicit realisations in the real case, that is, for the Heisenberg group over the real numbers. For instance in \cite{Nilpotentliegroups} the authors used a different approach, in such a way that all the Schrodinger representations $\pi_\lambda$ are obtained by composing the representations $\pi_1$ and $\pi_{-1}$ with the dilations $D_\lambda$. We can do a similar thing in the $p$-adic case with the difference that this time we cannot count with the existence of square roots. More precisely, after identifying $\widehat{\mathbb{H}}_d$ with $\Q_p^*$, we can find a finite subset $\widehat{G}_0$ of $\widehat{G}$ such that every $\pi_\lambda$ with $\lambda \in \Q_p^*$ is equivalent to $\pi_{\epsilon} \circ D_{\mu}$ for some $\epsilon \in \widehat{G}_0$ and $\mu \in \Q_p^*$. To see this let us point to the fact that any nonzero $p$-adic number $\lambda$ can be written in one of the following four ways: 
\begin{itemize}
    \item $\lambda = \mu^2$, $\mu \in \Q_p^*$,
    \item $\lambda = s_0 \mu^2$, $\mu \in \Q_p^*$ and $s_0 \in \Z_p \setminus p \Z_p$ is not a square, 

    \item $\lambda = s_0 \mu^2$, $\mu \in \Q_p^*$ and $s_0 \in \Z_p \setminus p \Z_p$ is not a square, 
    \item $\lambda = p s_0 \mu^2$, $\mu \in \Q_p^*$ and $s_0 \in \Z_p \setminus p \Z_p$ is not a square,
    \item $\lambda = p \mu^2$, $\mu \in \Q_p^*$.
\end{itemize}
Fix $L^2 (\Q_p^d)$ as our representation space and consider the linear maps $d_\mu : L^2 (\Q_p^d) \to L^2 (\Q_p^d)$, given by $d_\mu \varphi (u) := \varphi (\mu u )$, $\mu \in \Q_p^*$. Let $\epsilon \in \{ 1 , p , s_0 , p s_0 \},$ where $s_0 \in \Z_p \setminus p \Z_p$ is not a square. Then we can see that the composition of the representation $$\pi_\epsilon (\textbf{x} , \textbf{y}, z) \varphi (u) = e^{2 \pi i \{ \epsilon (z +  \textbf{x} \cdot \textbf{y} + \textbf{y} \cdot u) \}_p}  \varphi (u + x),$$with the dilation $D_\mu$ give us 
\begin{align*}
    (\pi_\epsilon \circ D_\mu )(\textbf{x} , \textbf{y}, z) \varphi (u) &= e^{2 \pi i \{ \epsilon ( \mu^2 z + \frac{\mu^2}{2} \textbf{x} \cdot \textbf{y} + \mu \textbf{y} \cdot u) \}_p}  \varphi (u + \mu x) \\ &= e^{2 \pi i \{ \epsilon \mu^2 (  z +  \textbf{x} \cdot \textbf{y} + \textbf{y} \cdot \frac{u}{\mu}) \}_p} d_\mu \varphi (\frac{u}{\mu} +  x) \\ &= (d_\mu^{-1} \circ \pi_{\epsilon \mu^2} (\textbf{x},\textbf{y},z)\circ d_\mu) \varphi (u).
\end{align*}This shows that $\pi_\epsilon \circ D_\mu$ and $\pi_{\epsilon \mu^2}$ are equivalent representations of $\mathbb{H}_d (\Q_p)$ on $L^2 (\Q_p^d)$. Also, we can obtain all the Schrodinger representations by composing the representations $\pi_\epsilon$, $\epsilon \in  \Q_p^* /(\Q_p^*)^2$, with the dilations of the group. 
\end{rem}
\begin{rem}\label{remsquareroot}
For $\mu, \epsilon$ and $\lambda$ as in the above remark we have the relation $ |\mu|_p = |\epsilon|_p^{-1/2} |\lambda|_p^{1/2} $.
\end{rem}
With the above realisation of the unitary dual of $\mathbb{H}_d$ the group Fourier transform takes the following particular form. Given a function $f \in L^2 (\mathbb{H}_d)$, its Fourier transform $\widehat{f}(\lambda):= \widehat{f}(\pi_\lambda)$ is the measurable field of bounded operators acting on each representation space $\mathcal{H}_\lambda = L^2(\Q_p^d)$ via the formula
\begin{align*}
    \widehat{f}(\lambda) \varphi (u) &:= \int_{\mathbb{H}_d} f(\textbf{x} , \textbf{y} ,z) \pi_\lambda^* (\textbf{x} , \textbf{y}, z) \varphi (u) d \textbf{x} d \textbf{y} dz \\ &= \int_{\Q_p^{2d+1}} f(\textbf{x} , \textbf{y} , z) e^{2 \pi i \{ \lambda( -z + \frac{1}{2} \textbf{x} \cdot \textbf{y} - \textbf{y} \cdot u)\}_p} \varphi (u - \textbf{x})d \textbf{x} d \textbf{y} dz \\ &=\int_{\Q_p^d \times \Q_p^d} \mathcal{F}_{\Q_p^{2d+1}}[f] ( \xi , \lambda(\frac{u+w}{2}) , \lambda)e^{2 \pi i \{ \xi (u-w)\}_p} \varphi (w) dw d \xi. 
\end{align*}For any $\lambda \in \Q_p^*$ the operator $\widehat{f} (\lambda)$ is a Hilbert-Schmidt operator with integral kernel $$K_{f , \lambda} (u,w) = \int_{\Q_p^d } \mathcal{F}_{\Q_p^{2d+1}}[f] (\xi , \lambda(\frac{u+w}{2}) , \lambda)e^{2 \pi i \{ \xi (u-w)\}_p}  d \xi.$$Its Hilbert-Schmidt norm is $$\| \widehat{f}(\lambda) \|_{HS} = |\lambda|_p^{-d/2} \| \mathcal{F}_{\Q_p^{2d+1}}[f] (\cdot ,\cdot, \lambda) \|_{L^2(\Q_p^{2d})},$$so the Plancherel formula takes the form $$\| f \|_{L^2(\mathbb{H}_d)}^2 = \int_{\Q_p^*} \| \widehat{f}(\lambda) \|_{HS}^2 | \lambda |_p^{d} d \lambda.$$Moreover, we can check that the following Fourier inversion formula: $$ f(x,y,z) = \int_{\Q_p^*} Tr[\pi_\lambda (x,y,z) \widehat{f}(\lambda)] |\lambda|_p^d d\lambda. $$
We are particularly interested in functions $f \in \mathcal{D}(\mathbb{H}_d)$ because of the nice properties of their Fourier transform. We collect some of them in the following proposition. 

\begin{pro}\label{prolocallyconstfuncHd}
For $f \in \mathcal{D}(\mathbb{H}_d)$ the following properties hold: 
\begin{itemize}
    \item[(i)] For every $\lambda \in \Q_p^*$ the operator $\widehat{f}(\lambda)$ has finite rank and therefore it is trace class.
    \item[(ii)] There is an $n_f \in \Z$ such that $\widehat{f}(\lambda) = 0$ for $|\lambda|_p > p^{n_f}.$
    \item[(iii)] If additionally $f \in \tilde{\mathcal{D}}(\mathbb{H}_d)$ then there is an $l_f \in \Z$ such that $\widehat{f}(\lambda) = 0$ for $|\lambda|_p < p^{l_f}$.
\end{itemize}
\end{pro}
We want to use the group Fourier transform to deal with a special class of left invariant operators that we call here \emph{pseudo-differential operators}. These are densely defined linear operators defined in terms of a \emph{symbol}, which we understand here as a measurable field of linear operators $$\sigma_T :=\{\sigma_T(\lambda) : \mathcal{D}'(\Q_p^d) \to \mathcal{D}'(\Q_p^d) \esp ; \esp \esp \lambda \in \Q_p^* \},$$such that for all $\lambda \in \Q_p^*$ it holds: $$\mathcal{H}_\lambda^\infty = \mathcal{D}(L^2 (\Q_p^d)) \subset Dom (\sigma_T(\lambda)):=\{\varphi \in \mathcal{H}_\lambda \esp : \esp \sigma_T(\lambda) \varphi \in \mathcal{H}_\lambda \}.$$
Given a symbol ${\sigma_T(\lambda)}_{\lambda \in \Q_p^*}$ we can define a densely defined left invariant operator $T$ that can be written in terms of its symbol $\sigma_T $ as $$T f(x,y,z) =  \int_{\Q_p^*} Tr[\pi_\lambda (x,y,z) \sigma_T (\lambda) \widehat{f}(\lambda)] |\lambda|_p^d d\lambda,$$where the symbol $\sigma_T $ can be "recovered" from the operator $T$ by means of the expression $$\sigma_T (\lambda): = \pi_\lambda^* (x,y,z) T \pi_\lambda (x,y,z)=T \pi_\lambda (0,0,0)= \mathcal{F}_{\mathbb{H}_d} [\textbf{k}](\lambda) = \widehat{\textbf{k}}_T(\lambda), $$where $\textbf{k}_T$ is the right convolution kernel of $T$. An advantage of this symbolic representation is that in case every operator $\sigma_T (\lambda)$ is invertible, and the inverses form a measurable field of bounded operators, a (formal) fundamental solution for $T$ is given by $$\langle E_T , f \rangle :=  \int_{\Q_p^*} Tr[  \widehat{ f \circ \iota}(\lambda) \sigma_T (\lambda)^{-1}] |\lambda|_p^d d\lambda, $$ where $f \circ \iota(x,y,z) :=  f((x,y,z)^{-1})$ and the above expression is well defined because by Proposition \ref{prolocallyconstfuncHd} the above integral is finite for any $f \in \tilde{\mathcal{D}}(\mathbb{H}_d)$. In particular, in the case where $T$ is a $\nu$-homogeneous operator of positive degree the invertibility of the symbol $\sigma_T$ reduces to its invertibility on a finite number of the representation spaces. The reason is that a pseudo-differential operator $T$ is homogeneous if and only if its associated symbol $\sigma_T$ has the following property: $$d_\gamma^{-1} \circ \sigma_T(\gamma^2 \lambda ) \circ d_\gamma = | \gamma |_p^\nu \sigma_T (\lambda) ,$$where $d_\gamma : L^2 (\Q_p^d) \to L^2 (\Q_p^d)$ is defined as $d_\gamma \varphi (u) : = \varphi ( \gamma u )$. Consequently, in order to check the invertibility of $\sigma_T$ on each representation space we only need to check the invertibility of $\sigma_T(\varepsilon)$ for $\varepsilon \in \Q_p^* /(\Q_p^*)^2$. We will ilustrate all this in the next subsection by using as an example the Vladimirov Laplacian on $G$.

\subsection{The Vladimirov Laplacian on the Heisenberg  group}
Finally we have collected all the necessary ingredients to fulfil the main goal of this section: to show the existence of a fundamental solution for the Vladimirov Laplacian $\MO^\alpha$, $\alpha >0$, on $\mathbb{H}_d$, and to give some estimate on its heat kernel. Recall that the Vladimirov Laplacian on $\mathbb{H}_d$ is the left invariant linear operator defined in terms of the VT directional operators as: $$\MO^\alpha = \sum_{j=1}^{d} \partial_{X_j}^\alpha + \partial_{Y_j}^\alpha + \partial_Z^{\alpha/2}.$$The operators $\partial^\alpha_{X_j} , \partial_{Y_j}^\alpha, \partial_Z^{\alpha}$ are left invariant pseudo-differential operators and their associated symbols are: \begin{itemize}
    \item \begin{align*}
       \sigma_{\partial_ {X_j}^\alpha} (\lambda) \varphi(u) := \partial_{X_j}^\alpha \pi_\lambda (0,0,0) \varphi (u) &= \frac{1 - p^\alpha}{1 - p^{- (\alpha + 1)}}\int_{\Q_p}  \frac{\varphi (u-u_j) - \varphi(u)}{|u_j|_p^{\alpha  + 1}} du_j \ = \partial_{u_j}^\alpha \varphi(u).
    \end{align*} 
    \item \begin{align*}
       \sigma_{\partial_{Y_j}^\alpha} (\lambda) \varphi(u) := \partial_{Y_j}^\alpha \pi_\lambda (0,0,0) \varphi (u) &= \frac{1 - p^\alpha}{1 - p^{- (\alpha + 1)}}\int_{\Q_p}  \frac{e^{2 \pi i \{ \lambda u_j y_j\}_p} - 1}{|y_j|_p^{\alpha  + 1}} dy_j \varphi (u) = |\lambda|_p^\alpha |u_j|^\alpha \varphi(u).
    \end{align*} 
    \item \begin{align*}
       \sigma_{\partial_Z^\alpha} (\lambda) \varphi(u) := \partial_Z^\alpha \pi_\lambda (0,0,0) \varphi (u) &= \frac{1 - p^\alpha}{1 - p^{- (\alpha + 1)}} \int_{\Q_p}  \frac{e^{2 \pi i \{ \lambda z\}_p} - 1}{|z|_p^{\alpha  + 1}} dz \varphi (u) =  |\lambda|_p^\alpha \varphi(u).
    \end{align*} 
\end{itemize}
In this way we can see that the symbol of $\MO^\alpha$ is given by $$\sigma_{\MO^\alpha} (\lambda) \varphi (u) = \big( \sum_{j=1}^d \partial_{u_j}^\alpha +  |\lambda|_p^\alpha |u_j|^\alpha \big) \varphi(u) + |\lambda|_p^{\alpha/2} \varphi (u),    $$which is a Schrodinger type operator like the ones studied by the authors in \cite[Chapter X]{Vladimirovbook}. We can apply their results to each operator $$\partial_{u_j}^\alpha +  |\lambda|_p^\alpha |u_j|^\alpha,$$to prove that each $\sigma_{\MO^\alpha} (\lambda) $is self adjoint in $\mathcal{H}_{\pi_\lambda} = L^2 (\Q_p^d)$ and invertible with compact inverse. We collect the results in \cite[Chapter X]{Vladimirovbook} in the following lemma. 

\begin{lema}\label{lemap-adicschrodinger}
Let $\sigma (v)$ and $V(x)$ be positive functions such that they tend to $+ \infty$ at infinity in $\Q_p.$ Then the operator $T_\sigma + V$ is positive, bounded from below and self adjoint on $L^2 (\Q_p)$, and its spectrum is discrete and it consist of real eigenvalues $\lambda_1 \leq \lambda_2 \leq ...,$  $\lambda_k \to \infty$ as $k \to  \infty$, where every eigenvalue has finite multiplicity. The corresponding eigenfunctions $\psi_k \in D(T_\sigma + V)$ form an orthonormal basis of $L^2 (\Q_p)$ and the following variational principle is valid: $$\lambda_k = \min_{\varphi_1,...,\varphi_k} \max\{ ((T_\sigma + V)\psi , \psi)_{L^2 (\Q_p)} \esp : \esp \psi \in Span\{ \varphi_1 ,..., \varphi_k \}, \esp \| \psi \|_{L^2 (\Q_p)} = 1  \}.$$ 
\end{lema}

As a corollary of the above lemma the operator $\sigma_{\MO^\alpha} (\lambda)$ is bounded from below and self adjoint on $L^2 (\Q_p^d)$ with compact inverse. Moreover, $\MO^\alpha$ is essentially self adjoint and for its self adjoint extension to $L^2(\mathbb{H}_d)$, here denoted by $\MO^\alpha_2$, the following version of the functional calculus applies. 

\begin{lema}\label{lemafuncalculHeisenlaplacian}
Let $E, E_\lambda$ be the spectral measures of $\MO^\alpha_2$ and $\sigma_{\MO^\alpha} (\lambda)$ respectively, so that $$\MO^\alpha_2= \int_\R s dE(s) \esp \esp \esp 
\text{and} \esp \esp \esp \sigma_{\MO^\alpha} (\lambda) =  \int_\R s dE_\lambda (s).$$Then for any borel subset $B \subset \R$, the orthogonal projection $E(B)$ is a bounded left invariant operator. The group Fourier transform of its convolution kernel $E(B) \delta_0$ is a measurable field of bounded operators given by $$\mathcal{F}_{\mathbb{H}_d} (E(B)\delta_e)(\pi_\lambda) = E_\lambda (B).$$Moreover, if $ \phi $ is a measurable function on $\R$, the spectral multiplier operator $\phi(\MO_2^\alpha)$ is defined  by $$ \phi( \MO_2^\alpha) := \int_{\R} \phi(s) d E(s),$$and its domain $Dom(\phi( \MO_2^\alpha))$ is the space of functions $ f \in L^2 (\mathbb{H}_d)$ such that $$\int_\R |\phi(s)|^2 (dE(s)f,f)< \infty.$$In particular, if $\phi \in L^\infty (\R)$ then $\phi(\MO_2^\alpha)$ is left invariant and bounded on $L^2(\mathbb{H}_d)$, and its convolution kernel, denoted by $\phi(\MO_2^\alpha) \delta_e$, is the unique distribution $\phi(\MO_2^\alpha) \delta_e \in \mathcal{D}(\mathbb{H}_d)'$ such that $$\phi(\MO_2^\alpha) f  = f * \phi(\MO_2^\alpha)  \delta_e,$$ for all $f \in \mathcal{D}(\mathbb{H}_d)$, and its group Fourier transform is $$\mathcal{F}\{ \phi(\MO_2^\alpha) \delta_e \}(\pi_\lambda) = \phi( \sigma_{\MO^\alpha} (\lambda)).$$
\end{lema}
By applying the above functional calculus to $\MO^\alpha$ with $\phi_t (s) = e^{-ts}$ we can check that the convolution kernel of the heat semigroup  $e^{-t \MO_2^\alpha}$, i.e. the heat kernel of $\MO^\alpha$, has the following explicit representation: 

\begin{equation}\label{heatkernellaplacian}
    h_{\MO^\alpha_2} ( t, x,y,z) = \int_{\Q_p^*} Tr[\pi_\lambda (x,y,z) e^{-t \sigma_{\MO^\alpha} (\lambda)}] |\lambda|_p^d d\lambda.
\end{equation}
We collect some properties of the above kernel in the following theorem. 

\begin{teo}\label{teoheatkernelproperties}
The heat kernel $h_{\MO^\alpha_2} $ associated to $\MO^\alpha$ has the following properties: 

\begin{itemize}
    \item[(i)] $h_{\MO^\alpha_2} (t, \cdot) *  h_{\MO^\alpha_2} (s, \cdot) = h_{\MO^\alpha_2} (t + s , \cdot )  $, for any $s,t>0$.
    \item[(ii)] $h_{\MO^\alpha_2} ( | \gamma |_p^\alpha t, \gamma x, \gamma y,\gamma^2 z) = |\gamma|^{-(2d+2)} h_{\MO^\alpha_2} ( t , x, y, z)$, for all $(x,y,z)\in \mathbb{H}_d$ and any $t>0$, $\gamma \in \Q_p^*$.
\item[(iii)] $h_{\MO^\alpha_2} ( t , x, y, z) = \overline{h_{\MO^\alpha_2} ( t , (x, y, z)^{-1})},$ for all $x \in \mathbb{H}_d$.  
\item[(iv)] The heat semigroup $e^{- t \MO_2^\alpha}$ is symmetric and Markovian. Moreover, the following estimate holds for its kernel: $$h_{\MO^\alpha_2} ( t , x, y, z) \asymp t^{-(2d+2)/\alpha} \prod_{j=1}^d \min\{1 , \frac{ t^{1 + 1/\alpha}}{ |x_j|_p^{1+\alpha}} \} \times \prod_{j=1}^d \min\{1 , \frac{ t^{1 + 1/\alpha}}{  |y_j|_p^{1+\alpha}} \} \times \min \{ 1 , \frac{ t^{1 + 2/\alpha}}{ |z|_p^{1+\alpha/2} }   \}.$$In particular for $(x,y,z) \in G_0 \setminus G_1$: $$ h_{\MO^\alpha_2} ( t , x, y, z) \asymp t^{-(2d+2)/\alpha}.$$
\end{itemize}
\end{teo} 
\begin{proof}
\esp
\begin{itemize}
    \item[(i)] Just notice that $$\mathcal{F}[h_{\MO^\alpha_2} (t, \cdot) *  h_{\MO^\alpha_2} (s, \cdot)](\lambda) = e^{-s \sigma_{\MO^\alpha} (\lambda)} e^{-t \sigma_{\MO^\alpha} (\lambda)} = e^{-(t+s) \sigma_{\MO^\alpha} (\lambda)}=\mathcal{F}[h_{\MO^\alpha_2} (t + s, \cdot)](\lambda).$$
    \item[(ii)]Using a simple change of variable and the $\alpha$-homogeneity of $\MO^\alpha$ we get: \begin{align*}
    h_{\MO^\alpha_2} ( | \gamma |_p^\alpha t, \gamma x, \gamma y,\gamma^2 z) &= \int_{\Q_p^*} Tr[\pi_\lambda (\gamma x,\gamma y, \gamma^2 z) e^{-t | \gamma |_p^\alpha \sigma_{\MO^\alpha} (\lambda)}] |\lambda|_p^d d\lambda \\ &=|\gamma^2|_p^{-(d+1)}\int_{\Q_p^*} Tr[ d_\gamma^{-1} \pi_{\lambda \gamma^2} (x,y,z) d_\gamma e^{-t d_\gamma^{-1} \sigma_{\MO^\alpha} (\lambda \gamma^2) d_\gamma}] |\gamma^2 \lambda|_p^d |\gamma^2|_p d\lambda \\& =
|\gamma|_p^{-(2d+2)}\int_{\Q_p^*} Tr[  \pi_{\lambda} (x,y,z)  e^{-t \sigma_{\MO^\alpha} (\lambda ) }] | \lambda|_p^d d\lambda\\ &= |\gamma|_p^{-(2d+2)} h_{\MO^\alpha_2} ( t,  x,  y, z). 
\end{align*}
  \item[(iii)] \begin{align*}
    \overline{h_{\MO^\alpha_2} ( t, (x,y,z)^{-1})} &= \int_{\Q_p^*}\overline{ Tr[\pi_\lambda (x,y,z)^{-1} e^{-t \sigma_{\MO^\alpha} (\lambda)}]} |\lambda|_p^d d\lambda \\ & =  \int_{\Q_p^*} Tr[\pi_\lambda (x,y,z) e^{-t \sigma_{\MO^\alpha} (\lambda)}] |\lambda|_p^d d\lambda = h_{\MO^\alpha_2} ( t,  x,  y, z).
\end{align*}\item[(iv)] This follows from the analysis for the Vladimirov Laplacian on $\Q_p^d$ done in \cite[Section 5.3]{IsoMarkovSemi}, see \cite[Proposition 5.13]{IsoMarkovSemi}.
 \end{itemize}
\end{proof}

As a consequence of the above lemma we can obtain the fundamental solution of $\MO^\alpha$ by integrating the heat kernel when $0 < \alpha < 2d+2.$  
\begin{coro}\label{corofundVladLaplacian1}
Let $0 < \alpha < 2d+2.$ Then a fundamental solution for the Vladimirov Laplacian exists and it defines an $ \alpha - (2d+2)$-homogeneous function given by $$\textbf{h}_{\MO^\alpha_2}(x,y,z) := \int_0^{\infty} h_{\MO^\alpha_2} ( t, x,y,z) dt.$$In consequence: $$\textbf{h}_{\MO^\alpha_2}(x,y,z) \asymp |(x,y,z)|_G^{\alpha - (2d+2)}.$$ 
\end{coro}

\begin{proof}
First, $\textbf{h}_{\MO^\alpha_2}$ is well defined because of the estimate on the heat kernel provided in Lemma \ref{teoheatkernelproperties}. Clearly $\textbf{h}_{\MO^\alpha_2}$ defines a fundamental solution for $\MO^\alpha$ so we only need to check the homogeneity property. Using Lemma \ref{teoheatkernelproperties} again and a change of variables we get:
\begin{align*}
    \textbf{h}_{\MO^\alpha_2} ( \gamma x, \gamma y,\gamma^2 z) &= \int_0^\infty h_{\MO^\alpha_2} ( t, \gamma x, \gamma y,\gamma^2 z) dt \\ &= |\gamma|_p^\alpha \int_0^\infty h_{\MO^\alpha_2} ( |\gamma|_p^{\alpha}(|\gamma|_p^{-\alpha}t), \gamma x, \gamma y,\gamma^2 z) |\gamma|_p^{-\alpha} dt \\ &=  |\gamma|_p^{\alpha} \int_0^\infty  h_{\MO^\alpha_2} ( |\gamma|_p^{\alpha} t, \gamma x, \gamma y,\gamma^2 z) dt \\ &= |\gamma|_p^{\alpha - (2d + 2)}\int_0^\infty  h_{\MO^\alpha_2} ( t,  x,  y, z) dt \\ &= |\gamma|_p^{\alpha - (2d + 2)} \textbf{h}_{\MO^\alpha_2} ( x, y, z).
\end{align*}
\end{proof}
Another consequence of the functional calculus for $\MO_2^\alpha$ and Theorem \ref{teoheatkernelproperties} is that fractional powers of $\MO_2^\alpha$ and their fundamental solutions exist and they are well defined homogeneous distributions. 

\begin{coro}\label{cororiezpotential}
Let $\beta\in \C$ such that $0< \mathfrak{Re}(\beta) < 2d+2$. Then the linear operator $L^\beta:=(\MO_2^\alpha)^{\beta/\alpha}$ defined via functional calculus possesses a fundamental solution, which is an $\beta - (2d+2)$ homogeneous distribution, determined by the Riesz potential $$\mathcal{I}_\beta (x,y,z) := \frac{1}{\Gamma(\beta/\alpha)} \int_0^\infty t^{\beta/\alpha - 1} h_{\MO^\alpha_2} ( t, x, y, z)dt.$$ 
\end{coro}
To conclude this section let us remark that for the case $\alpha \geq 2d+2$, it is also possible to obtain a fundamental solution in the formal sense because for that we only need the invertibility of the symbol. The same is true for the operator $$\MO^\alpha_{sub} = \sum_{j=1}^{d} \partial_{X_j}^\alpha + \partial_{Y_j}^\alpha , \esp \esp \alpha>0,$$that we call here the "Vladimirov sub-Laplacian". The symbol of the above $\alpha$-homogeneous operator is  $$\sigma_{\MO^\alpha_{sub}} (\lambda) \varphi (u) = \big( \sum_{j=1}^d \partial_{u_j}^\alpha +  |\lambda|_p^\alpha |u_j|^\alpha \big) \varphi(u),$$and by Lemma \ref{lemap-adicschrodinger} a functional calculus like in Lemma \ref{lemafuncalculHeisenlaplacian} exists for this operator. However, without the estimates on the heat kernel of the operator we can only construct a formal fundamental solution
    $$\langle E_{\MO^\alpha_{sub}} , f \rangle :=  \int_{\Q_p^*} Tr[  \widehat{ f \circ \iota}(\lambda) \sigma_{\MO^\alpha_{sub}} (\lambda)^{-1}] |\lambda|_p^d d\lambda, $$ where $f \circ \iota(x,y,z) :=  f((x,y,z)^{-1})$, so that a very interesting problem would be to provide some heat kernel estimates for $\MO^\alpha_{sub}$. More generally, since our analysis here to obtain the fundamental solution of $\MO^\alpha$ work in principle for graded $\K$-Lie groups, an interesting problem in this setting is to prove the existence of a fundamental solution for more general homogeneous operators. This has been done already for real Lie groups with ideas very similar to the ones we use here but there are some technical difficulties that one must overcome. Anyway, by now we restrict our attention to the Vladimirov Laplacian on some groups like $\mathbb{H}_d$ and $\mathbb{E}_4$ where we have all the necessary ingredients for our proofs.   

\subsection{The group $\mathbb{E}_4$}
To conclude our work we want to give a quick review of the way our arguments work for the Engel group $\mathbb{E}_4$. The analysis is the same as for the Heisenberg group so we only need to precise the details that are different from our exposition on $\mathbb{H}_d$. The Engel group $\mathbb{E}_4$ introduced in \ref{exaEngelalgebra} is a graded group for which the Kirillov theory applies and we can obtain explicitly the unitary dual $\widehat{\mathbb{E}}_4$ of $\mathbb{E}_4$ in terms of the orbits $\mathcal{O}$ of the co-adjoint action of $\mathbb{E}_4$. There is a Borel measure on $\widehat{ \mathbb{E}}_4$ which assigns measure zero to all the co-adjoint orbits except for those corresponding with the representations  $$\pi_{\lambda , \mu} (x,y_1,y_2,y_3) \varphi (u) = e^{2 \pi i \{ -\frac{\mu}{2 \lambda} y_1 + \lambda y_3 -\lambda y_2 u + \frac{\lambda}{2} y_1 u^2 \}_p} \varphi(u + x),$$
so that we can identify $\widehat{ \mathbb{E}}_4$ with $\Q_p^4$. With this identification, the group Fourier transform for functions in $L^2(\mathbb{E}_4)$ takes the form \begin{align*}
    \widehat{f}(\lambda , \mu) \varphi (u):=\widehat{f}(\pi_{\lambda , \mu}) \varphi (u) &= \int_{\mathbb{E}_4} f(x,y_1,y_2,y_3) \pi_{\lambda , \mu} (x,y_1,y_2,y_3) \varphi (u) dx \\&= \int_{\Q_p} \Big(\int_{\Q_p} \mathcal{F}_{\Q_p^4}[f]( \xi , \frac{\lambda}{2}v^2 - \frac{\mu}{2 \lambda}, - \lambda v , \lambda) e^{2 \pi i \{ (u - v) \xi \}_p} \Big) \varphi (v) dv,
\end{align*}and in particular for functions $f \in \mathcal{D}(\mathbb{E}_4)$ we have the following properties:

\begin{pro}\label{prolocallyconstfuncE4}
For $f \in \mathcal{D}(\mathbb{E}_4)$ the following properties hold: 
\begin{itemize}
    \item[(i)] For every $\lambda, \mu \in \Q_p$, the operator $\widehat{f}(\pi_{\lambda , \mu})$ has finite rank and therefore it is trace class.
    \item[(ii)] There is an $n_f \in \Z$ such that $\widehat{f}(\pi_{\lambda , \mu}) = 0$ for $\|(\lambda, \mu)\|_p > p^{n_f}.$
    \item[(iii)] If additionally $f \in \tilde{\mathcal{D}}(\mathbb{E}_4)$, then there is an $l_f \in \Z$ such that $\widehat{f}(\lambda) = 0$ for $\|(\lambda, \mu)\|_p < p^{l_f}$.
\end{itemize}
\end{pro}

For the group $\mathbb{E}_4$ the VT directional operators $\partial^\alpha_{X} , \partial_{Y_j}^\alpha$ are left invariant pseudo-differential operators and their associated symbols are: \begin{itemize}
    \item \begin{align*}
       \sigma_{\partial_ {X}^\alpha} (\lambda) \varphi(u) := \partial_{X}^\alpha \pi_\lambda (0,0,0) \varphi (u) &= \frac{1 - p^\alpha}{1 - p^{- (\alpha + 1)}}\int_{\Q_p}  \frac{\varphi (u-u_j) - \varphi(u)}{|u_j|_p^{\alpha  + 1}} du_j \ = D_u^\alpha \varphi(u).
    \end{align*} 
    \item \begin{align*}
       \sigma_{\partial_{Y_1}^\alpha} (\lambda) \varphi(u) := \partial_{Y_1}^\alpha \pi_\lambda (0,0,0) \varphi (u) &= \frac{1 - p^\alpha}{1 - p^{- (\alpha + 1)}}\int_{\Q_p}  \frac{e^{2 \pi i \{ y_1 (\frac{\lambda}{2} u^2 - \frac{\mu}{2 \lambda} )) \}_p} - 1}{|y_j|_p^{\alpha  + 1}} dy_1 \varphi (u) \\ &= |\frac{\lambda}{2} u^2 - \frac{\mu}{2 \lambda}|_p^\alpha  \varphi(u).
    \end{align*} 
    \item \begin{align*}
        \sigma_{\partial_{Y_2}^\alpha} (\lambda) \varphi(u) := \partial_{Y_2}^\alpha \pi_\lambda (0,0,0) \varphi (u) &= |\lambda u|_p^\alpha \varphi (u).
    \end{align*}
    \item \begin{align*}
       \sigma_{\partial_{Y_3}^\alpha} (\lambda) \varphi(u) := \partial_{Y_3}^\alpha \pi_\lambda (0,0,0) \varphi (u) &= \frac{1 - p^\alpha}{1 - p^{- (\alpha + 1)}} \int_{\Q_p}  \frac{e^{2 \pi i \{ \lambda y_3 \}_p} - 1}{|y_3|_p^{\alpha  + 1}} dy_3 \varphi (u) =  |\lambda|_p^\alpha \varphi(u).
    \end{align*} 
\end{itemize}
Hence, the Vladimirov Laplacian on $\mathbb{E}_4$ $$\MO^\alpha = \partial_{X}^\alpha + \partial_{Y_1}^\alpha + \partial_{Y_2}^{\alpha/2} + \partial_{Y_3}^{\alpha/3}, \esp \esp \esp \alpha>0,$$ is a left invariant $\alpha$-homogeneous pseudo-differential operator with symbol 
$$ \sigma_{\MO^\alpha}(\lambda , \mu)\varphi(u) = (D_u^\alpha + |\frac{\lambda}{2} u^2 - \frac{\mu}{2 \lambda}|_p^\alpha + |\lambda u|_p^{\alpha/2} +  |\lambda|_p^{\alpha/3}) \varphi(u), $$and by using again Lemma \ref{lemap-adicschrodinger} we can check that $\sigma_{\MO^\alpha}(\lambda , \mu)$ defines an invertible self-adjoint operator with compact inverse in each representation space. Thus, we can associate to $\MO_2^\alpha$, the self-adjoint extension of $\MO^\alpha$, a functional calculus and we can obtain the following theorem: 

\begin{teo}\label{teoheatkernelproperties}
The heat kernel $h_{\MO^\alpha_2} $ associated to $\MO^\alpha$ has the following properties: 

\begin{itemize}
    \item[(i)] $h_{\MO^\alpha_2} (t, \cdot) *  h_{\MO^\alpha_2} (s, \cdot) = h_{\MO^\alpha_2} (t + s , \cdot )  $, for any $s,t>0$.
    \item[(ii)] $h_{\MO^\alpha_2} ( | \gamma |_p^\alpha t, \gamma x, \gamma y_1,\gamma^2 y_2, \gamma^3 y_3) = |\gamma|^{-7} h_{\MO^\alpha_2} ( t , x, y_1,y_2,y_3)$, for all $(x,y_1,y_2 , y_3)\in \mathbb{E}_4$ and any $t>0$, $\gamma \in \Q_p^*$.
\item[(iii)] $h_{\MO^\alpha_2} ( t , x, y_1, y_2 ,y_3) = \overline{h_{\MO^\alpha_2} ( t , (x, y_1,y_2, y_3)^{-1})},$ for all $(x, y_1,y_2, y_3) \in \mathbb{E}_4$.  
\item[(iv)] The heat semigroup $e^{- t \MO_2^\alpha}$ is symmetric and Markovian. Moreover, the following estimate holds for its kernel:$$ h_{\MO^\alpha_2} ( t , x, y_1,y_2,y_3) \asymp t^{-7/\alpha}.$$
\end{itemize}
\end{teo} 
\begin{coro}\label{corofundVladLaplacian1}
Let $0 < \alpha < 7.$ Then a fundamental solution for the Vladimirov Laplacian exists and it defines an $ \alpha - 7$-homogeneous function given by $$\textbf{h}_{\MO^\alpha_2}(x,y_1,y_2,y_3) := \int_0^{\infty} h_{\MO^\alpha_2} ( t, x,y_1 , y_2 , y_3) dt.$$
Consequently: $$\textbf{h}_{\MO^\alpha_2}(x,y_1 , y_2 , y_3) \asymp |(x,y_1 , y_2 , y_3)|_{\mathbb{E}_4}^{\alpha - 7}.$$ 
\end{coro}

\begin{coro}\label{cororiezpotential}
Let $\beta\in \C$ such that $0< \mathfrak{Re}(\beta) < 7$. Then the linear operator $L^\beta:=(\MO_2^\alpha)^{\beta/\alpha}$ defined via functional calculus possesses a fundamental solution, which is an $\beta - 7$ homogeneous distribution, determined by the Riesz potential $$\mathcal{I}_\beta (x,y_1, y_2 , y_3) := \frac{1}{\Gamma(\beta/\alpha)} \int_0^\infty t^{\beta/\alpha - 1} h_{\MO^\alpha_2} ( t, x,y_1, y_2 , y_3)dt.$$ 
\end{coro}
\nocite{*}
\bibliographystyle{acm}
\bibliography{main}
\Addresses

\end{document}